\documentclass[letterpaper, 10 pt, conference]{ieeeconf}  

\IEEEoverridecommandlockouts                              

\let\labelindent\relax
\usepackage{amsmath,amssymb,amsfonts, amsthm}
\usepackage{enumitem}
\usepackage{wrapfig}
\usepackage{subcaption}
\usepackage{multirow}
\usepackage{graphicx}
\usepackage{subcaption}
\usepackage{algorithm, algorithmic}
\usepackage{mathrsfs}
\usepackage{hyperref}
\usepackage[noadjust]{cite}

\newcommand{\Ss}{\bar{\mathbb E}}

\usepackage{xcolor}

\newcommand{\res}[1]{\textcolor{black}{#1}}
\usepackage{environ}
\NewEnviron{rese}{%
    \textcolor{black}{\BODY}%
}

\newtheorem{theorem}{Theorem}
\newtheorem{lemma}[theorem]{Lemma}
\newtheorem{proposition}[theorem]{Proposition}

\newtheorem{assumption}{Assumption}
\newtheorem{remark}{Remark}
\newtheorem{definition}{Definition}

\newtheorem{problem}{Problem}

\title{\LARGE \bf
 Conditional Generative Modeling of Stochastic LTI Systems: A Behavioral Approach 
}
\author{Jiayun Li$^{1}$ and Yilin Mo$^{1}$
\thanks{$^{1}$Jiayun Li and Yilin Mo are with the Department of Automation and BNRist, Tsinghua University, Beijing, P.R.China. Emails: \texttt{lijiayun22@mails.tsinghua.edu.cn, ylmo@tsinghua.edu.cn}.}%
}

\begin{document}

\maketitle
\thispagestyle{empty}
\pagestyle{empty}

\begin{abstract}
This paper presents a data-driven model for Linear Time-Invariant (LTI) stochastic systems by sampling from the conditional probability distribution of future outputs given past input-outputs and future inputs. It operates in a fully behavioral manner, relying solely on the current trajectory and \res{pre-collected input-output data}, without requiring explicit identification of system parameters. We refer to this model as a behavioral Conditional Generative Model (CGM). We prove the convergence of the distribution of samples generated by the CGM as the size of the \res{trajectory library} increases, with an explicit characterization of the convergence rate. Furthermore, we demonstrate that the gap between the asymptotic distribution of the proposed CGM and the true posterior distribution obtained by Kalman filter, which leverages the knowledge of all system parameters and all historical data, decreases exponentially with respect to the length of past samples. Finally, we integrate this generative model into predictive controllers for stochastic LTI systems. Numerical results verify the derived bounds and demonstrate the effectiveness of the controller equipped with the proposed behavioral CGM.
\end{abstract}


\section{Introduction}
Linear Time-Invariant (LTI) systems play a crucial role in control due to their simplicity and wide applicability in various domains, such as control design~\cite{morari_model_1999,bemporad_robust_1999}, time-series forecasting~\cite{mahalakshmi2016survey}, and fault detection~\cite{isermann2005model, venkatasubramanian2003review}. These models must be accurate, especially in demanding scenarios such as agile locomotion control~\cite{kaufmann2023champion, neunert2018whole}, and must account for system uncertainties for safety-critical applications~\cite{knight2002safety}. Although first-principle models, i.e., models derived from physical laws, exist for many systems, they may not suffice due to factors such as deviation from the nominal parameters and/or inaccuracy in the uncertainty models~\cite{coulson2019data,coulson_distributionally_2021}. In contrast, input-output (I/O) data of the system are often readily available, which has recently spurred significant research into the data-driven construction of precise predictive models for unknown LTI systems~\cite{coulson2019data}.

Classic data-driven modeling approaches primarily use data to identify parameters of parametric models, i.e., state-space models and transfer functions~\cite{ljung1998system, ho1966effective, VANOVERSCHEE199475, VERHAEGEN199461}. Recently, \emph{behavioral models}, a subset of nonparametric models, have emerged as a viable alternative to parametric models for LTI systems~\cite{coulson2019data,VERHEIJEN2023100914,dorfler_bridging_2023,sader2023causality, WANG2025111897}. \res{Unlike system identification, which requires selecting a model structure and estimating system parameters, behavioral approaches define dynamical systems directly as sets of trajectories~\cite{MARKOVSKY202142}. These methods first construct a \res{\emph{trajectory library}} by pre-collecting input-output data, and during online operation, they use past samples and future inputs to predict future outputs consistent with the trajectory library, eliminating the need for explicit parametric identification.} \res{This representation-free perspective avoids model-structure selection errors~\cite{dorfler_bridging_2023,MARKOVSKY202142}. Moreover, unlike two-stage approaches that first identify a parametric model and then design a controller, behavioral models support a direct data-to-controller design, which mitigates error propagation across multi-stage procedures and simplifies implementation in practical systems~\cite{dorfler_bridging_2023,coulson2019data}.}

For \emph{deterministic} LTI systems, the Willems' lemma~\cite{fundamental_lemma}, a classic result in behavioral system theory, serves as an exact behavioral model. It states that any feasible trajectory of the system can be represented as a linear combination of trajectories from the \res{trajectory library}~\cite{coulson2019data}. Building on this lemma, methods such as DeePC~\cite{coulson2019data} further construct direct data-driven controllers by formulating an optimization problem akin to Model Predictive Control (MPC), replacing the parametric predictive models with behavioral models. These controllers are extended to noisy systems by incorporating regularization terms into the objective function~\cite{coulson2019data, MATTSSON2023625,dorfler_bridging_2023,coulson_distributionally_2021}. However, these regularization terms may not faithfully represent the stochastic nature of the system, since the Willems' lemma, which these methods rely on, applies strictly to deterministic systems. This gap motivates the development of a behavioral model for \emph{stochastic} LTI systems.

A substantial amount of recent studies in data-driven control have focused on extending Willems' lemma to stochastic dynamical systems that takes system uncertainties into account. One research direction incorporates additive noise, in addition to inputs and outputs, into the system trajectory, enabling the Willems' lemma to be applied for predicting future outputs and noise based on past data. Pan et al.~\cite{pan_stochastic_2023,pan_towards_nodate} leverage this extended lemma to develop direct data-driven predictive controllers for stochastic LTI systems, while Kerz et al.~\cite{kerz_data-driven_2023} utilize the extended lemma to construct tube-based predictive controllers. However, the extended lemma relies on disturbances in the collected trajectories, which is typically unobservable and requires additional estimation~\cite{pan_stochastic_2023}. Teutsch et al.~\cite{teutsch2024sampling} address this issue for systems with bounded noise through behavior-consistent noise sampling. In a different approach, Wang et al.~\cite{WANG2025111897} extends the Willems' lemma by applying it to the expectation update of the steady-state Kalman filter. However, the expectation alone does not capture the complete distribution of the system outputs, which is crucial for methods such as robust control~\cite{KWAKERNAAK1993255}.

This paper proposes a behavioral CGM that extends the Willems' lemma~\cite{fundamental_lemma} to stochastic LTI systems. The model \res{captures the stochastic nature of LTI systems by sampling from a data-consistent conditional distribution} of future outputs, given past I/O samples and future inputs. \res{Conditional Generative Models (CGMs) have evolved from early conditional GANs~\cite{mirza2014conditional} and conditional VAEs~\cite{sohn2015cvae} to become fundamental tools for machine learning~\cite{ajay2022conditional}. Multiple advanced methods are built upon CGM foundations: score-based diffusion models learn conditional score functions for iterative denoising~\cite{song2021score}, while flow matching approaches train conditional normalizing flows by regressing vector fields of conditional probability paths~\cite{lipman2023flow}. CGMs have \res{also} demonstrated success in diverse applications, from image generation~\cite{mirza2014conditional,sohn2015cvae} to decision-making~\cite{ajay2022conditional}. In this work, we adapt the CGM framework to stochastic LTI systems, avoiding explicit system identification while providing uncertainty quantification.}

The contribution of the paper is threefold:
\begin{itemize}
    \item We propose a behavioral CGM \res{that generates samples from a data-consistent conditional distribution of future outputs, given past input-output data and prescribed future inputs. The model avoids explicit parametric identification while maintaining online computational efficiency.}
    \item We establish the theoretical foundations of the proposed CGM through a two-fold convergence analysis: first proving the asymptotic convergence of generated sample distributions with explicitly quantified rates as the \res{trajectory library} size increases, and then showing that this asymptotic distribution exponentially approaches the optimal Kalman filter prediction as the length of historical data grows.
    \item \res{We develop a stochastic predictive controller based on the proposed generative modeling framework. 
    Numerical results show that the controller achieves reduced constraint violations and competitive control cost, while achieving data and computational efficiency.}
\end{itemize}

\subsubsection*{Paper Structure}
In Section~\ref{sec:problem_formulation}, we formulate the problem of constructing a behavioral CGM for stochastic LTI systems. Section~\ref{sec:method} proposes a behavioral CGM for stochastic LTI systems. Section~\ref{sec:performance_analysis} and Section~\ref{sec:nonGaussian} analyze the performance of the proposed CGM for systems with Gaussian and non-Gaussian noise, respectively. In Section~\ref{sec:control}, we incorporate the proposed model into the predictive control framework. Finally, Section~\ref{sec:simulations} presents numerical results on a three-pulley system and Section~\ref{sec:conclusions} concludes the paper.

\subsubsection*{Notations}
$\mathbb{R}$ is the set of all real numbers. $\mathbb{R}^{a \times b}$ is the set of $a \times b$ real matrices. $I_n$ denotes the $n$-dimensional unit matrix. $\mathbf{0}_n$ denotes an $n$-dimensional all-zero column vector and $\boldsymbol{0}_{m\times n}$ represents an $m\times n$ zero matrix. $\boldsymbol{1}_n$ and $\boldsymbol{1}_{m\times n}$ are similarly defined. $A^\top$ denotes the matrix transpose. $\|A\|_2$ denotes the $2$-norm for a vector $A$ or a matrix $A$. For Hermitian matrices $A\in\mathbb{C}^{n\times n}$, $A\succeq 0$ implies that $A$ is positive semi-definite. The extended expectation $\Ss$ is defined as $\bar{\mathbb E}_t(\phi_t)=\lim_{t\to\infty}\mathbb E(\phi_t)$ for \res{an appropriate} $\phi$.

\res{Following control theory conventions, all lowercase variables with time subscripts ($x_t, y_t, u_t, \phi_t$) represent the system's time-indexed signals. 
In our stochastic setting, unless otherwise specified, both these quantities and their samples are modeled as \emph{random vectors}, which is consistent with the 
standard probabilistic viewpoint adopted in large-sample analyses such as the law of 
large numbers~\cite{chung2000course}.}

\section{Problem Formulation}\label{sec:problem_formulation}
Consider an $n$-dimensional LTI system with $m$ inputs and $p$ outputs taking the following form:
\begin{align}\label{eq:linear_system}
    x_{t+1}&=Ax_{t}+Bu_{t}+w_{t},\nonumber \\
    y_{t}&=Cx_{t}+v_{t},
\end{align}
where $w_{t}$ and $v_{t}$ represent i.i.d. Gaussian process and observation noise with zero mean and covariances $\res{Q\succeq 0}, R\succ 0$, respectively. Without loss of generality, we assume that $(A, B, C)$ is a \res{minimal} realization of the system~\eqref{eq:linear_system}\res{, which implies that the system is both controllable and observable.} Let $\ell$ denote the lag~\cite{coulson2019data} of the system. All parameters $A, B, C, Q, R$ are treated as unknown parameters\footnote{The value of $A,\,B,\,C,\,Q,\,R$ are only used to characterize the performance of the proposed CGM, and are absent in the construction of the CGM.}.

When the system~\eqref{eq:linear_system} is noise-free, i.e., $w_t=0, v_t=0$, the input-output relationship of the system can be exactly represented by a behavioral model, known as the Willems' lemma~\cite{fundamental_lemma}. Specifically, suppose we are provided with a \res{trajectory library} $\mathcal D$, an offline dataset representing the system's behavior, consisting of a single trajectory:
\[\mathcal D=\{\tilde u_{t}, \tilde y_t\}_{t=1}^K.\]
Let $\tilde u=\mathrm{col}(\tilde u_1, \cdots, \tilde u_{K}), \tilde y=\mathrm{col}(\tilde y_1, \cdots, \tilde y_K)$. Moreover, let
$\mathscr H_{T_{ini}+T}(\tilde u)$ denote the following Hankel matrix formed from $\tilde u$:
\[\mathscr H_{T_{ini}+T}(\tilde u)=\begin{bmatrix}
    \tilde u_1 & \cdots & \tilde u_{K-T_{ini}-T+1} \\
    \vdots & \ddots & \vdots \\
    \tilde u_{T_{ini}+T} & \cdots & \tilde u_{K}
\end{bmatrix},\]
and $\mathscr H_{T_{ini}+T}(\tilde y)$ is similarly defined. Moreover, define
\begin{align}
    \begin{bmatrix} \tilde U_p \\ \tilde U_f\end{bmatrix}\triangleq\mathscr H_{T_{ini}+T}(\tilde u), \begin{bmatrix} \tilde Y_p \\ \tilde Y_f \end{bmatrix}\triangleq\mathscr H_{T_{ini}+T}(\tilde y),\label{eq:Yf_def_fundamental_lemma}
\end{align}
where $\tilde U_p$ consists of the first $T_{ini}$ block rows of $\mathscr H_{T_{ini}+T}(\tilde u)$, and $\tilde U_f$ consists of the last $T$ block rows of $\mathscr H_{T_{ini}+T}(\tilde u)$. $\tilde Y_p$ and $\tilde Y_f$ are similarly defined. The subscript $p$ and $f$ represents past and future trajectories respectively, and the subscript $ini$ stands for initial.
\begin{assumption}[\res{Persistency of Excitation}]\label{assump:persistently_exciting}
    \res{The given input-output trajectory $\{\tilde u_t, \tilde y_t\}_{t=1}^{K}$ has length
        $K \ge (m+1)(T_{ini}+T+n)-1$. Moreover, the input sequence $\tilde u$ is persistently exciting of order at least $T_{ini}+T+n$, meaning that 
        $\mathscr H_{T_{ini}+T+n}(\tilde u)$ has full row rank.}
\end{assumption}

Based on the \res{trajectory library} $\mathcal D$ of the noise-free system~\eqref{eq:linear_system}, consider any trajectory $\{u_t, y_t\}_{t=1}^{T_{ini}+T}$ of length $T_{ini}+T$ from the system. For simplicity, we assume that the current time step is $T_{ini}$, which allows us to separate the trajectory into past and future components:
\begin{align}
    u_{ini}&=\mathrm{col}(u_{1}, \cdots, u_{T_{ini}}), y_{ini}=\mathrm{col}(y_{1}, \cdots, y_{T_{ini}}),\nonumber \\
    u_f&=\mathrm{col}(u_{T_{ini}+1}, \cdots, u_{T_{ini}+T}),\nonumber \\
    y_f&=\mathrm{col}(y_{T_{ini}+1}, \cdots, y_{T_{ini}+T}).\label{eq:uini_yini_uf_yf_def}
\end{align}
Here, $u_{ini}, y_{ini}$ represent the past I/O samples, and $u_f, y_f$ correspond to the future inputs and outputs of the system. Then, the Willems' lemma states the following result:
\begin{proposition}[Willems' lemma~\cite{coulson2019data,fundamental_lemma}]\label{prop:fundamental_lemma}
   \res{Assume that the input-output data $\{\tilde u_t,\tilde y_t\}_{t=1}^K$ satisfy Assumption~\ref{assump:persistently_exciting}
    and that $T_{ini} \ge \ell$. Let $(u_{ini},y_{ini},u_f,y_f)$ be any length-$(T_{ini}+T)$ input-output trajectory that is generated by the system~\eqref{eq:linear_system}.
    Then there exists a vector $g\in\mathbb R^{K-T_{ini}-T+1}$ such that}
    \begin{equation}\label{eq:fundamental_lemma}
        \begin{bmatrix} \tilde U_p \\ \tilde Y_p \\ \tilde U_f \\ \tilde Y_f\end{bmatrix}g=\begin{bmatrix} u_{ini} \\ y_{ini} \\ u_f \\ y_f\end{bmatrix}.
    \end{equation}
    Moreover, if $g$ \res{satisfies}:
    \begin{align}
        \begin{bmatrix} \tilde U_p \\ \tilde Y_p \\ \tilde U_f\end{bmatrix}g&=\begin{bmatrix} u_{ini} \\ y_{ini} \\ u_f\end{bmatrix},\label{eq:fundamental_lemma_1}
    \end{align}
    \res{then the associated future output is uniquely determined as $y_f=Y_fg$.}
\end{proposition}

\begin{rese}
Proposition~\ref{prop:fundamental_lemma} shows that, under suitable assumptions, the future outputs $y_f$ of a deterministic LTI system are \emph{uniquely determined} from past inputs and outputs together with prescribed future inputs. This deterministic characterization reduces the prediction of future outputs to solving a linear equation. However, this determinism no longer holds for the stochastic system~\eqref{eq:linear_system}, since the outputs become random vectors driven by the noise processes. Hence, the value of $y_f$ is not uniquely determined but varies across realizations, and the relevant object becomes the probability distribution of $y_f$ given past information and future inputs.

To formalize this, we consider a probability space
\((\Omega,\mathcal{G},\mathbb{P}_u)\), where the subscript $u$
indicates that the probability measure depends on the prescribed
input sequence \(\{u_t\}\). The process noise \(w_t\) and
measurement noise \(v_t\) are modeled as random vectors on
this space, and the outputs
\(y_t\) are random vectors on
\((\Omega,\mathcal{G},\mathbb{P}_u)\) as measurable functions 
of the noises and the inputs.

The central object of interest is the conditional distribution
\begin{equation}\label{eq:true_conditional}
    p_u(y_f \mid \mathcal{F}^y_{T_{ini}}),
\end{equation}
where $\mathcal{F}^y_{T_{ini}}=\sigma(y_t\mid t\leq T_{ini})$ is the $\sigma$-algebra generated by past outputs up to time $T_{ini}$. For notational simplicity, and with a slight abuse of notation, we equivalently write~\eqref{eq:true_conditional} as
\begin{equation}\label{eq:cond_simplified}
    p_{\mathrm{sys}}(y_f \mid \mathcal{F}_{T_{ini}}, u_f),
\end{equation}
where $\mathcal{F}_{T_{ini}}=\sigma(u_{t},y_{t}\mid t\leq T_{ini})$ includes both past inputs and outputs, and $u_f$ are prescribed future inputs. We refer to~\eqref{eq:true_conditional}--\eqref{eq:cond_simplified} as the \emph{true stochastic system model}.

\begin{remark}
    The key quantity for system modeling is the conditional distribution $p_{\mathrm{sys}}(y_{f} \mid \mathcal{F}_{T_{ini}}, u_f)$, which characterizes how future outputs depend probabilistically on past observations and future inputs. 
    Since this distribution is generally difficult to characterize without perfect knowledge of system parameters, the literature has proposed various approximations, including statistical moments, series-expansion coefficients, and sampled trajectories, i.e., sampling from $p_{\mathrm{sys}}(y_{f} \mid \mathcal{F}_{T_{ini}}, u_f)$~\cite{faulwasser_behavioral_2023,dorfler_bridging_2023,WANG2025111897,kladtke_towards_2025}. 
    In this paper, we adopt the sampled-trajectory representation, as it naturally connects to many control design methods such as scenario-based MPC.
\end{remark}

Motivated by this viewpoint, we introduce the concept of a Conditional Generative Model (CGM). 
The role of a CGM is to approximate the true stochastic system model~\eqref{eq:cond_simplified} by generating samples from an approximate conditional distribution $q$.

\begin{definition}[Conditional Generative Model (CGM)]\label{def:cgm}
A CGM for an LTI system is an algorithm that, given past information $\mathcal{F}_{T_{ini}}$ and prescribed future inputs $u_f$, generates random samples of the future output sequence
\[
    \bar y_f \sim q(\cdot \mid \mathcal{F}_{T_{ini}}, u_f),
\]
where $q$ is a conditional distribution that approximates the true distribution $p_{\mathrm{sys}}(y_f \mid \mathcal{F}_{T_{ini}}, u_f)$ in~\eqref{eq:cond_simplified}.
\end{definition}
Note that Definition~\ref{def:cgm} formalizes a CGM at an abstract level. The assumptions on the available system information depend on the particular realization of the framework. 

Ideally, when system noise is Gaussian and all system parameters $(A, B, C, Q, R)$ are known, a Kalman filter initialized with the true initial state distribution fully characterizes the Gaussian posterior distribution $p_{\mathrm{sys}}(y_f\mid \mathcal F_{T_{ini}}, u_f)$ through the exact mean and covariance predicted by the filter. Sampling from this Gaussian distribution yields a CGM with $q(\bar y_f\mid \mathcal F_{T_{ini}}, u_f) = p_{\mathrm{sys}}(y_f\mid \mathcal F_{T_{ini}}, u_f)$, which represents the theoretical performance upper bound.

In contrast, we propose a \emph{behavioral CGM} that requires neither prior knowledge of the system parameters nor explicit parameter identification. The only assumption is that the noise is known to be Gaussian distributed\footnote{We relax this assumption to noise with zero mean and bounded second moment in Section~\ref{sec:nonGaussian}.}, while the system matrices $(A,B,C)$, the noise covariances $(Q,R)$, and the initial state distribution are all treated as unknown.
The algorithm relies solely on pre-collected input-output trajectory data~$\mathcal{D}$ together with the current information $(u_f,\mathcal{F}_{T_{ini}})$ to predict future outputs $y_f$.

Compared with the ideal Kalman filter benchmark, our approximation of $p_{\mathrm{sys}}$ is constrained by two practical challenges: 
(i) the trajectory library $\mathcal{D}$ has finite size, and 
(ii) only a finite past horizon $(u_{ini},y_{ini})$ is available instead of the full information set $\mathcal{F}_{T_{ini}}$. 
These constraints induce a gap between the CGM distribution 
$q(\bar y_f \mid u_{ini}, y_{ini}, u_f; \mathcal{D})$ and the true stochastic system model $p_{\mathrm{sys}}(y_f \mid \mathcal{F}_{T_{ini}}, u_f)$. 
The objective of this paper is to design a behavioral CGM that narrows this gap and to establish guarantees quantifying how the gap decreases with the trajectory library size and the truncation length $T_{ini}$.

\end{rese}

\subsection{\res{Trajectory} Library Construction}\label{subsec:data_collection}
Before introducing the proposed behavioral CGM and \res{analyzing} its performance, we introduce additional assumptions regarding the \res{trajectory library} $\mathcal D$.

We assume that $\mathcal D$ either contains \emph{one single long trajectory} generated by system~\eqref{eq:linear_system} or \emph{multiple independent trajectories of the same length}. Moreover, we assume that the inputs $u_t$ in \eqref{eq:linear_system} are generated by a linear controller:
\begin{assumption}[Control Inputs in \res{the Trajectory Library}]\label{assump:input_exciting}
    The inputs in the \res{trajectory library} are generated by the following \res{$n_{\mathrm{ctrl}}$-dimensional} stabilizing controller:
    \begin{align}
        \phi_{t+1}&=\res{A_{\mathrm{ctrl}}}\phi_t+\res{B_{\mathrm{ctrl}}}y_t, \nonumber \\
        u_t&=\res{C_{\mathrm{ctrl}}}\phi_t+\res{\nu_{t}}, \label{eq:stabilizing_controller}
    \end{align}
    where \res{$\nu_{t}$} is i.i.d. Gaussian noise with zero mean and covariance $\res{R_{\mathrm{ctrl}}\succ 0}$. We assume that the internal state $\phi_t$ is known and the initial state $\phi_1\sim\mathcal N(0, \Sigma_\phi)$ where $\Sigma_\phi\succ 0$. We further assume that $(\res{A_{\mathrm{ctrl}}, B_{\mathrm{ctrl}}, C_\mathrm{ctrl}})$ is a minimum realization of the controller~\eqref{eq:stabilizing_controller} without loss of generality \res{and $(A_{\mathrm{ctrl}}, B_{\mathrm{ctrl}})$ has controllability index $\kappa_{\mathrm{ctrl}}$}.
\end{assumption}
\begin{remark}
    The controller described in~\eqref{eq:stabilizing_controller} encompasses most linear control schemes with state or output feedback. For instance, white noise inputs \res{are obtained by setting $u_t = \nu_t$. Such white noise excitation is commonly used in system identification~\cite{ljung2007practical,van_den_hof_system_2005,hof_modelling_2000}. PID controller can also be represented in state-space form~\eqref{eq:stabilizing_controller} with appropriate parameter choices.}
\end{remark}

If the \res{trajectory library} consists of multiple independent trajectories, we assume, without loss of generality, that each trajectory has the same length \(T_{ini} + T\), and the number of trajectories satisfies \(N \geq m(T_{ini}+T)+T_{ini}p+n\). We then arrange the data into the following matrices, with each column representing a trajectory:
\[\res{\check U^{(N)}}=\begin{bmatrix}
    \check u_1^{(1)} & \cdots & \check u_1^{(N)} \\
    \vdots & \ddots & \vdots \\
    \check u_{T_{ini}+T}^{(1)} & \cdots & \check u_{T_{ini}+T}^{(N)}
\end{bmatrix},\]
\[ \res{\check Y^{(N)}}=\begin{bmatrix}
    \check y_1^{(1)} & \cdots & \check y_1^{(N)} \\
    \vdots & \ddots & \vdots \\
    \check y_{T_{ini}+T}^{(1)} & \cdots & \check y_{T_{ini}+T}^{(N)}
\end{bmatrix},\]
\begin{equation}
    \res{\check\Phi^{(N)}}=\begin{bmatrix}
    \check\phi_1^{(1)} & \cdots & \check\phi_1^{(N)} 
\end{bmatrix},\label{eq:Phi_def_multiple}
\end{equation}
where the superscript \((i)\) indicates the \(i\)-th trajectory, and the check mark signifies that the samples are drawn from multiple independent trajectories. \res{$\check\Phi^{(N)}$} consists of the first internal state of the controller in each trajectory.

Similar to the \res{Willems'} lemma, we also partition the matrices into past and future blocks by denoting the first $T_{ini}$ block rows of \res{$\check U^{(N)}$} and \res{$\check Y^{(N)}$} as \res{$\check U_p^{(N)}$} and \res{$\check Y_p^{(N)}$}, and the last $T$ block rows as \res{$\check U_f^{(N)}$} and \res{$\check Y_f^{(N)}$}, respectively. For convenience, we concatenate \res{$\check U_p^{(N)}, \check Y_p^{(N)}, \check U_f^{(N)}$} into a single matrix:
\begin{equation}\label{eq:Z_Yf_def_multiple}
    \res{\check Z^{(N)} = \mathrm{col}(\check U_p^{(N)}, \check Y_p^{(N)}, \check U_f^{(N)}).}
\end{equation}
The \res{trajectory library} is then denoted as:
\[
\res{\mathcal{\check D}^{(N)} = \{\check \Phi^{(N)},\check Z^{(N)}, \check Y_f^{(N)}\}.}
\]
\begin{remark}
    In the context of the closed-loop controller described in Assumption~\ref{assump:input_exciting}, incorporating the internal state $\phi_t$ of the controller into the \res{trajectory library} is crucial \res{because the CGM must remove the influence of the offline controller 
    to recover the correct conditional distribution}, as demonstrated in Section~\ref{sec:performance_analysis}. On the other hand, if the inputs within the \res{trajectory library} are open-loop, \res{$\Phi^{(N)}$} can be excluded from the library.
\end{remark}

For the single-trajectory case (as discussed in Section~\ref{sec:problem_formulation}), similar matrices $\tilde U_p, \tilde U_f, \tilde Y_p,\tilde Y_f$ are constructed as Hankel matrices with the assumption that the length of the trajectory $K\geq (m+1)(T_{ini}+T)+T_{ini}p+n-1$. Next, we concatenate the internal states of the controller $\tilde\phi_t$ that generate the first inputs $\tilde u_i$ in the $i$-th column of $\tilde U$ into:
\begin{align}
    \tilde\Phi=\begin{bmatrix} \tilde\phi_1 & \cdots & \tilde\phi_{K-T_{ini}-T+1} \end{bmatrix}.\label{eq:Phi_def_single}
\end{align}
For simplicity of the subsequent discussions, we assume that the width of the Hankel matrices $\tilde U_p,\tilde U_f, \tilde Y_p, \tilde Y_f$ is identical to that of \res{$\check U^{(N)}, \check Y^{(N)}$} in the multi-trajectory library \res{$\mathcal{\check D}^{(N)}$}, i.e., 
\begin{equation}\label{eq:width_assump}
    K-T_{ini}-T+1=N,
\end{equation}
and denote the corresponding matrices with superscript $N$. Then, we concatenate \res{$\tilde U_p^{(N)}, \tilde Y_p^{(N)}, \tilde U_f^{(N)}$} into a single matrix:
\begin{equation}\label{eq:Z_def_single}
    \res{\tilde Z^{(N)}=\mathrm{col}(\tilde U_p^{(N)}, \tilde Y_p^{(N)}, \tilde U_f^{(N)}).}
\end{equation}
Thus, the \res{trajectory library} with a single trajectory is denoted as:
\[\res{\mathcal{\tilde D}^{(N)}=\{\tilde \Phi^{(N)}, \tilde Z^{(N)}, \tilde Y_f^{(N)}\}.}\]

To unify the \res{notation} for both behavior libraries, we introduce \res{$Z^{(N)}$} to represent either \res{$\check Z^{(N)}$} or \res{$\tilde Z^{(N)}$}. Similarly, \res{$\check Y_f^{(N)}$} and \res{$\tilde Y_f^{(N)}$} are unified as \res{$Y_f^{(N)}$}, and \res{$\check\Phi^{(N)}$} and \res{$\tilde\Phi^{(N)}$} are unified as \res{$\Phi^{(N)}$}. Consequently, the unified notation for both behavior libraries is:
\begin{align}
    \res{\mathcal D^{(N)}=\{\Phi^{(N)}, Z^{(N)}, Y_f^{(N)}\},}\label{eq:behavior_library}
\end{align}
which captures the information of the stochastic system~\eqref{eq:linear_system}. 

In the following sections, based on \res{$\mathcal D^{(N)}$}, we propose a behavioral CGM for the stochastic LTI system~\eqref{eq:linear_system} and analyze its performance.

\section{Conditional Generative Model}\label{sec:method}
This section presents a behavioral CGM for the stochastic LTI system~\eqref{eq:linear_system}. Specifically, consider the trajectory $(u_{ini}, y_{ini}, u_f, y_f)$ formulated in Section~\ref{sec:problem_formulation}. At the current time step $T_{ini}$, define the known \emph{initial trajectory} as:
\begin{equation}\label{eq:z_def}
   z_{T_{ini}} = \mathrm{col}(u_{ini}, y_{ini}, u_{f}).
\end{equation}
Then, to sample $\bar y_f$ based on $z_{T_{ini}}$ and the \res{trajectory library $\mathcal D^{(N)}$}, we propose to randomly sample $\alpha$ \res{from the solution set}:
\begin{align}
   \res{\left\{\,\alpha\in\mathbb R^N \;\middle|\;
   \underbrace{\begin{bmatrix} \Phi^{(N)} \\ Z^{(N)}\end{bmatrix}}_{\Xi^{(N)}}\alpha
   = \begin{bmatrix} 0 \\ z_{T_{ini}}\end{bmatrix}
   \right\}},\label{eq:alpha_equation}
\end{align}
and the generated future outputs are:
\[\bar y_f=\res{Y_f^{(N)}}\alpha.\]
To sample $\alpha$, we \res{propose to} decompose $\alpha$ into a special solution part and a \res{general solution} part:
\begin{align}
   \alpha&=(\arg\min_g\|g\|_2^2\mid \res{\Xi^{(N)}}g=\mathrm{col}(0, z_{T_{ini}}))\nonumber \\
   &\qquad\quad+(\beta\sim\mathcal N(0, \frac{1}{N}I_N)\mid \res{\Xi^{(N)}}\beta=0).\label{eq:alpha_decomposition}
\end{align}
\res{The first term is the special solution that satisfies the equality constraint with minimum norm. Under Gaussian noise models, it corresponds to the maximum-likelihood estimate and serves as the deterministic baseline. The second term exploits the null space degrees of freedom to model conditional uncertainty. The Gaussian sampling is consistent with the system's 
noise model, while the $1/N$ variance scaling ensures that the resulting covariance 
converges at the correct rate as the trajectory library grows. Section~\ref{sec:performance_analysis} further 
shows that these constructions provide asymptotically correct conditional means and 
covariances for the generated $\check y_f$.}

\vspace{-\baselineskip}%
\begin{rese}
    \begin{remark}
         Inspired by the Willems' lemma for noise-free systems, the equality constraint in the second row block of~\eqref{eq:alpha_equation} ensures that the linear combination of trajectories in $\mathcal{D}^{(N)}$ matches the given initial trajectory, while the constraint $\Phi^{(N)}\alpha = 0$ eliminates the influence of offline controller states, ensuring the generated distribution reflects open-loop system dynamics rather than inconsistent information specific to the data collection phase.
    \end{remark}
\begin{remark}\label{rem:xi_full_rank}
   We prove in Appendix~\ref{appendix:xi_full_rank} that under Assumption~\ref{assump:input_exciting} and when $N\geq (r+1)(T_{ini}+T)+r\kappa_{\mathrm{ctrl}}$ with $r=n_{\mathrm{ctrl}}+m(T_{ini}+T)+pT_{ini}$, $\Xi^{(N)}$ has full row rank almost surely. As a result, the equality constraint in~\eqref{eq:alpha_equation} is feasible and both components of $\alpha$ in~\eqref{eq:alpha_decomposition} are well-defined.
\end{remark}
\end{rese}
\vspace{-\baselineskip}%
\begin{remark}\label{rem:interpretation}
   The special solution component in~\eqref{eq:alpha_decomposition} is given by $\Xi^{(N)\dagger}\res{\zeta_{T_{ini}}}$. \res{On the other hand, the random part $(\beta\sim\mathcal{N}(0, \frac{1}{N}I_N)\mid \Xi^{(N)}\beta=0)$ represents a conditional random vector. Since $\beta$ is Gaussian and $\Xi^{(N)}\beta$ is a linear transformation of $\beta$, $(\beta, \Xi^{(N)}\beta)$ is jointly Gaussian. Therefore, the conditional distribution of $\beta$ given $\Xi^{(N)}\beta=0$ is well-defined and equivalent in distribution to $(I-\Xi^{(N)\dagger}\Xi^{(N)})\beta'$}, where $\beta'\sim\mathcal{N}(0, \frac{1}{N}I_N)$. 
    \res{As a result, the deterministic component of the CGM output is given by
    \(
    \bar y_{f, det} = \Theta_f^{(N)} z_{T_{ini}}, 
    \text{with}\ \Theta_f^{(N)} = Y_f^{(N)}\Xi^{(N)\dagger}\mathrm{col}(0, I_{T_{\mathrm{ini}}(m+p)+Tm}),
    \)
    and the stochastic component is
    \(
    \bar\beta = Y_f^{(N)}(I-\Xi^{(N)\dagger}\Xi^{(N)})\beta',\beta'\sim\mathcal{N}(0,\tfrac{1}{N}I_N).
    \)}
\end{remark}
\res{The procedure of the proposed CGM is summarized in Algorithm~\ref{alg:nonparametric_generative_model}. 
In accordance with the decomposition in~\eqref{eq:alpha_decomposition}, the algorithm consists of a deterministic minimum-norm solution consistent with the initial trajectory, 
and a stochastic perturbation drawn from the null space of $\Xi^{(N)}$ to represent conditional 
uncertainty.}

\begin{algorithm}
\caption{Computational Procedure for Conditional Generative Model of Stochastic LTI Systems \res{(see Remark~\ref{rem:interpretation} for definitions of notations such as $\Theta_f^{(N)}$ and $\bar{\beta}$)}.}
\label{alg:nonparametric_generative_model}
\begin{algorithmic}[1]
\REQUIRE Trajectory library \res{$\mathcal D^{(N)}$}, initial trajectory $u_{ini}, y_{ini}, u_f$
\STATE \res{Construct $\Xi^{(N)}$} according to~\eqref{eq:alpha_equation} and \res{$Y_f^{(N)}$} according to~\eqref{eq:Yf_def_fundamental_lemma} for $\mathcal{\tilde D}^{(N)}$ or ~\eqref{eq:Phi_def_multiple} for $\mathcal{\check D}^{(N)}$
\STATE $\res{\Theta_f^{(N)}}\leftarrow \res{Y_f^{(N)}\Xi^{(N)\dagger}}\mathrm{col}(0, I_{T_{ini}(m+p)+Tm})$
\STATE $z_{T_{ini}}\leftarrow\mathrm{col}(u_{ini}, y_{ini}, u_f)$
\STATE $\res{\bar y_{f, det}}\leftarrow \res{\Theta_f^{(N)}} z_{T_{ini}}$ \res{// Deterministic prediction}
\STATE Sample $\beta \sim \mathcal N(0, \frac{1}{N}I_N)$
\STATE $\res{\bar\beta}\leftarrow \res{Y_f^{(N)}(I-\res{\Xi^{(N)\dagger}\Xi^{(N)}})}\beta$
\STATE $\res{\bar y_f}\leftarrow \res{\bar y_{f, det}+\bar\beta}$ \res{// Add randomness for uncertainty}
\RETURN $\res{\bar y_f}$
\end{algorithmic}
\end{algorithm}
\vspace{-0.2cm}

\begin{remark}
	To incorporate the CGM into predictive controllers, the CGM is required to generate samples online in real-time. To make this computation efficient, steps \res{1, 2, 5, 6} in Algorithm~\ref{alg:nonparametric_generative_model} can be performed offline, while the online computation of $\res{\check y_f}$ involves only steps \res{3, 4, 7}. This includes a single matrix-vector multiplication $\Theta z_{T_{ini}}$ in line \res{4}, with complexity $O(p(m+p)TT_{ini})$, and a vector addition in line \res{7}, with complexity $O(Tp)$. Hence, the overall online computation complexity \res{for generating $T$ future outputs} is $O(p(m+p)TT_{ini})$.
\end{remark}

We have proposed a behavioral CGM for the stochastic LTI system~\eqref{eq:linear_system}. In the following section, we derive theoretical guarantees on the performance of the proposed CGM.

\section{Performance Analysis}\label{sec:performance_analysis}
This section analyzes the performance of the proposed behavioral CGM. First, we characterize the conditional distribution of samples $\bar y_f$ generated by the CGM. Then, in Section~\ref{subsec:convergence}, \res{we analyze the convergence of this distribution as the trajectory library size $N$ increases and characterize the resulting asymptotic distribution.} Finally, in Section~\ref{subsec:optimality_gap}, we assess the gap between this asymptotic distribution and the true posterior distribution generated from the Kalman filter.
\begin{assumption}\label{assump:multiple_trajectories}
    \begin{enumerate}
        \item The initial state of the current online trajectory is Gaussian distributed.\label{assump:current_traj_initial_state}
        \setcounter{enumi}{2}
        \begin{enumerate}[label=\arabic{enumi}\alph*)]
            \item For the multi-trajectory behavior library \res{$\mathcal{\check D}^{(N)}$}, the initial state $\check x_1^{(i)}$ of each trajectory is independently sampled from the same Gaussian distribution, $\check x_{1}^{(i)}\sim \mathcal N(0, \check \Sigma)$, where $\check\Sigma\succ 0$. Moreover, we assume that $\check\Sigma-\check\Sigma_{x\phi}\Sigma_\phi^\dagger\check\Sigma_{x\phi}^\top\succ 0$, where $\check\Sigma_{x\phi}=\mathbb E(\check x_1^{(1)}\check\phi_1^{(1)\top})$ and $\Sigma_\phi$ is introduced in Assumption~\ref{assump:input_exciting}.\label{assump:initial_state}
            \item For the single-trajectory behavior library \res{$\mathcal{\tilde D}^{(N)}$}, the initial state $\tilde x_1$ of the trajectory is Gaussian distributed.
        \end{enumerate}
    \end{enumerate}
\end{assumption}

The following proposition characterizes the distribution of the generated $\bar y_f$ using Algorithm~\ref{alg:nonparametric_generative_model}:
\begin{proposition}\label{prop:model_distribution_single}
    The conditional distribution of the generated $\res{\bar y_{f}}$ \res{given $z_{T_{ini}},\mathcal D^{(N)}$} satisfies:
    \begin{align}
        &\mathbb E(\res{\bar y_{f}}\mid z_{T_{ini}}; \res{\mathcal{D}^{(N)}})=\res{\Theta_f^{(N)}} z_{T_{ini}}, \nonumber\\
        &\res{\Sigma_f^{(N)}}\triangleq\mathrm{Cov}(\res{\bar y_{f}}\mid z_{T_{ini}}; \res{\mathcal{D}^{(N)}}) \nonumber \\
        &=\frac{1}{N}\res{Y_f^{(N)}(I-\Xi^{(N)\dagger} \Xi^{(N)})(I-\Xi^{(N)\dagger} \Xi^{(N)})^\top Y_f^{(N)\top}}.\label{eq:generated_distribution}
    \end{align}
\end{proposition}
\res{The result follows directly from Remark~\ref{rem:interpretation}, which provides explicit expressions for both components in~\eqref{eq:alpha_decomposition}.}
In the following subsection, we analyze the convergence properties of both the conditional expectation coefficient \res{$\Theta_f^{(N)}$} and the conditional covariance \res{$\Sigma_f^{(N)}$} as \res{trajectory} library size $N\to\infty$.

\subsection{Convergence of Conditional Distribution}\label{subsec:convergence}
In this subsection, we analyze the convergence of the distribution of the generated $\bar y_f$ as $N$ tends to infinity. 

\res{To derive the explicit form of the asymptotic distribution, we start from the ideal full-information case in which a Kalman filter (KF) has access to the true system parameters $(A,B,C,Q,R)$ and the initial state distribution. 
In this setting, the KF provides the true posterior distribution of future outputs $p_{\mathrm{sys}}(y_f \mid \mathcal F_{T_{ini}}, u_f)$~\cite{anderson2005optimal}.}
\begin{rese}
\begin{proposition}\label{prop:kf_distribution}
    Consider the KF initialized with the true initial state distribution 
    $x_1 \sim \mathcal N(\mu,P)$.
    Conditioned on $(u_{ini},y_{ini},u_f)$,  
    the distribution of $y_f$ generated by system~\eqref{eq:linear_system} 
    is Gaussian with mean
    \begin{align}
        \mathbb E\!\left(y_f \mid z_{T_{ini}}\right)
        = \eta_f(P)\, z_{T_{ini}} 
        + O_f\, \hat\psi_{T_{ini}+1}^-(P)\, \mu,
    \end{align}
    and covariance
    \begin{align}
        \mathcal Y_f(P)
        &\triangleq \mathrm{Cov}\left(y_f \mid z_{T_{ini}}\right)\nonumber \\
        &= O_f P_{T_{ini}+1}^-(P) O_f^\top
        + G_f Q_f G_f^\top
        + R_f. \label{eq:conditional_distribution}
    \end{align}
    Here, \(\eta_f: \mathbb{R}^{n\times n} \to \mathbb{R}^{Tp \times (T_{ini}(m+p)+Tm)}\),
    \(\hat\psi_{T_{ini}+1}^-: \mathbb{R}^{n\times n} \to \mathbb{R}^n\),
    and \(P_{T_{ini}+1}^-: \mathbb R^{n\times n} \to \mathbb R^{n\times n}\)
    are linear operators that depend on the initial covariance \(P\).
    The matrices \(O_f,\, G_f,\, Q_f,\, R_f\) are determined by the system parameters 
    \((A,B,C,Q,R)\).
\end{proposition}
\end{rese}
\res{The proof of this proposition is provided in Appendix~\ref{append:kf_distribution}.}

The following theorem establishes the convergence of the conditional expectation coefficient \res{$\Theta_f^{(N)}$} and conditional covariance \res{$\Sigma_f^{(N)}$} from Proposition~\ref{prop:model_distribution_single} to their counterparts
$\eta_f(\Sigma), \mathcal Y_f(\Sigma)$ derived in Proposition~\ref{prop:kf_distribution} as $N\to\infty$.

\begin{theorem}[Asymptotic convergence of expectation and covariance of $\bar y_f$]\label{thm:convergence_N_multiple}
   Under Assumption\res{~\ref{assump:persistently_exciting}-\ref{assump:multiple_trajectories}}, \res{let $\epsilon>0$ be arbitrary. Then, as $N\to\infty$,} the coefficients in the conditional expectation and the conditional covariance of the generated \res{$\check y_{f}$} in~\eqref{eq:generated_distribution} satisfy:
    \begin{align}
        &\lim_{N\to\infty}\frac{\|\res{\Theta_f^{(N)}}-\eta_{f}(\Sigma)\|}{N^{-0.5+\epsilon}}=0\ \textrm{a.s.},\label{eq:theta_asymptotic_N_multiple} \\
        &\lim_{N\to\infty}\frac{\|\res{\Sigma_f^{(N)}}-\mathcal{Y}_f(\Sigma)\|}{N^{-0.5+\epsilon}}=0\ \textrm{a.s.},\label{eq:offline_asymptotic_N}
    \end{align}
    where $\Sigma$ is defined as follows:
    \begin{itemize}
        \item For the multi-trajectory library \res{$\mathcal D^{(N)}$},
        \begin{align}
            \Sigma=\check\Sigma-\check\Sigma_{x\phi}\Sigma_\phi^{\dagger}\check\Sigma_{x\phi}^\top,\label{eq:Sigma_def_multiple}
        \end{align}
        where $\Sigma_\phi$ is defined in Assumption~\ref{assump:input_exciting} and $\check\Sigma, \check\Sigma_{x\phi}$ are introduced in Assumption~\ref{assump:multiple_trajectories}.
        \item For the single-trajectory library \res{$\mathcal D^{(N)}$},
        \begin{align}
            \Sigma=\tilde\Sigma_{xx}-\tilde\Sigma_{x\phi}\res{\tilde\Sigma_{\phi\phi}^{\dagger}}\tilde\Sigma_{x\phi}^\top,\label{eq:Sigma_def_single}
        \end{align}
        where $\tilde\Sigma_{xx}=\mathbb{\bar E}_t(\tilde x_t\tilde x_t^\top), \tilde\Sigma_{x\phi}=\mathbb{\bar E}_t(\tilde x_t\tilde\phi_t^\top),\tilde\Sigma_{\phi\phi}=\mathbb{\bar E}_t(\tilde\phi_t\tilde\phi_t^\top)$.
    \end{itemize}
\end{theorem}
The proof of this theorem is provided in Appendix~\ref{append:convergence_N}. \res{The proof uses a martingale convergence result to establish almost-sure convergence of 
the sample covariance matrices, followed by a differentiability argument that 
preserves convergence rates.}

Theorem~\ref{thm:convergence_N_multiple} not only establishes the convergence of the conditional distribution of the generated $\bar y_f$, but also demonstrates its asymptotic equivalence to \res{the posterior distribution of $y_f$ generated by KF in~\eqref{eq:conditional_distribution}}, when the KF is initialized with $x_1 \sim \mathcal N(0, \Sigma)$. Notice that the true KF would use all past information, including $u_{-\infty:0},\,y_{-\infty:0}$ to infer the true posterior distribution of $x_1$. On the other hand, due to the finite length of the trajectories in the \res{trajectory} library, our CGM can only leverage information collected starting from time $1$, which create a mismatch between our CGM and the \res{true KF}. However, we shall show in the next subsection, that this discrepancy decays exponentially with respect to $T_{ini}$ and hence is negligible in practice.

\subsection{Optimality Gap with KF}\label{subsec:optimality_gap}
In this subsection, we quantify the discrepancy between the asymptotic distribution of $\bar{y}_f$ and the distribution predicted by the correctly initialized KF. 

We begin by deriving the posterior distribution of $y_f$ from the true KF. Recall from Section~\ref{sec:problem_formulation} that the behavioral model relies solely on the most recent $T_{ini}$ I/O samples, while the trajectory may extend beyond this finite window. In contrast, we assume that the KF has access to the entire historical trajectory, $u_{-\infty:T_{ini}}, y_{-\infty:T_{ini}}$, along with the initial state distribution $\mathcal N(\mu_{\mathcal T}, \Sigma_{\mathcal T})$. \res{Given the linear system~\eqref{eq:linear_system} and the initial state, the KF recursion is~\cite{anderson2005optimal}:
\begin{align}
    &\hat x_{t}=\hat x_t^-+K_t(y_t-C\hat x_t^-),\quad P_t=(I-K_tC)P_t^-,\nonumber \\
    &\hat x_{t+1}^-=A\hat x_t+Bu_t,\quad P_{t+1}^-=AP_tA^\top+Q,\label{eq:kf}
\end{align}
with $K_t=P_t^-C^\top(CP_t^-C^\top+R)^{-1}$}.

\begin{proposition}[Predicted Distribution of KF]\label{prop:optimal_distribution} 
    Under Assumptions~\ref{assump:multiple_trajectories}~\ref{assump:current_traj_initial_state}), \res{consider the KF~\eqref{eq:kf} initialized with 
    $(\mu_{\mathcal{T}}, \Sigma_{\mathcal{T}})$
    and updated recursively using the past samples 
    $u_{-\infty:0}, y_{-\infty:0}$. Let $\hat x_1^{-}$ and $P_1^{-}$ denote the estimate of~\eqref{eq:kf} at time $t=1$.
    Conditioned on $(\mathcal F_{T_{ini}}, u_f)$,  
    the distribution of $y_f$ generated by the system~\eqref{eq:linear_system} 
    is Gaussian with mean}
    \begin{align} 
        &\hat{y}_f^* = \eta_f(P_1^-)z_{T_{ini}}+O_f\hat\psi_{T_{ini}+1}^-(P_1^-)\hat x_1^-, \nonumber 
    \end{align} 
    \res{and covariance}
    \begin{align}
        &\mathcal{Y}_f^* = O_f P_{T_{ini}+1}^{-}(P_1^-) O_f^\top + G_f Q_f G_f^\top + R_f. \label{eq:optimal_distribution} 
    \end{align}
\end{proposition}
\res{Similar to Proposition~\ref{prop:kf_distribution}, the forms in this proposition follow directly from the linear structure of the recursive equations of KF.}

The distribution given by the KF in Proposition~\ref{prop:optimal_distribution} can be interpreted in two phases:
\begin{enumerate}
    \item From $t = -\infty$ to $0$: Estimate the distribution $(\hat{x}_1^-, P_1^-)$ of $x_1$ using the input-output samples $u_{-\infty:0}, y_{-\infty:0}$ and the initial state distribution $(\mu_{\mathcal{T}}, \Sigma_{\mathcal{T}})$;
    \item From $t = 1$ to $T_{ini}$: First estimate $x_{T_{ini}+1}$ using $u_{ini}, y_{ini}$ as in~\eqref{eq:kf}, and then predict $y_f$ using $u_f$ and the estimated $x_{T_{ini}+1}$ as described in~\eqref{eq:conditional_distribution}.
\end{enumerate}
Notably, the predictions in~\eqref{eq:optimal_distribution} coincide with the asymptotic distribution of the CGM in Theorem~\ref{thm:convergence_N_multiple}, with the only difference being initializations on the distribution of $x_1$: $(0, \Sigma)$ for the asymptotic distribution of the CGM and $(\hat{x}_1^-, P_1^-)$ for the KF.

The following theorem establishes that 
the gap between the asymptotic distribution and the distribution given by the KF decreases exponentially as the length of the initial trajectory $T_{ini}\to\infty$. The proof of the theorem is provided in \res{Appendix~\ref{append:offline_asymptotic_Tini}}.
\begin{theorem}[Gap with the optimal generative model]\label{thm:offline_asymptotic_Tini}
    \res{Under Assumptions~\ref{assump:persistently_exciting}--\ref{assump:multiple_trajectories}, and suppose $(A, Q^{1/2})$ is stabilizable.
    Let $\Sigma$ and $P_1^-$ be any positive definite matrices.
    As the initial horizon length $T_{ini}$ tends to infinity, 
    the conditional mean and covariance of the generated samples $\check y_f$ 
    converge exponentially to those of the optimal generative model 
    described in Proposition~\ref{prop:optimal_distribution}.  
    More precisely, there exists a constant $0 < \rho < 1$ such that}
    \begin{align}
        &\lim_{T_{ini}\to\infty}\|\eta_{f}(\Sigma)-\eta_f(P_1^-)\|\rho^{-T_{ini}}=0, \nonumber \\
        &\lim_{T_{ini}\to\infty}\|O_f\hat\psi_{T_{ini}+1}^-(P_1^-)\|\rho^{-T_{ini}}=0\nonumber \\
        &\lim_{T_{ini}\to\infty}\|\mathcal{Y}_f(\Sigma)-\mathcal Y_f(P_1^-)\|\rho^{-T_{ini}}=0.
    \end{align}
\end{theorem}

\section{Extension to Systems with Non-Gaussian Noise}\label{sec:nonGaussian}
The convergence analysis in Section~\ref{sec:performance_analysis} is specifically developed for systems with Gaussian noise, where the Bayesian optimal predictor reduces to the closed-form KF that our CGM asymptotically recovers. For systems with non-Gaussian noise, \res{the Bayesian optimal estimator becomes nonlinear and typically lacks closed-form expressions}. \res{Although nonlinear Bayesian estimators can in principle approximate such posteriors, commonly used methods such as particle filters ofen suffer from computational complexity, particle degeneracy, and the curse of dimensionality~\cite{Elfring2021ParticleFA}.} Consequently, linear filtering approaches remain widely adopted in practice for their robustness and efficiency, even in non-Gaussian settings. \res{Motivated by this}, we next establish \res{the optimality of our} CGM among all linear predictors for systems with non-Gaussian noise. We first introduce the following assumptions:
\begin{assumption}\label{assump:nonGaussian}
    \begin{enumerate}
        \item \res{The random variables $\{w_t\}$ and $\{v_t\}$ form i.i.d. sequences with zero mean and 
        finite second moments.
        Moreover, the joint distribution of $(\phi_1,\nu_0,v_0)$ has support with nonempty interior.}
        \item For the multi-trajectory behavior library \res{$\mathcal{\check D}^{(N)}$}, the initial state $\check x_1^{(i)}$ of each trajectory is independently sampled from the same distribution with zero mean and covariance $\check\Sigma\succ 0$. Moreover, we assume that $\check\Sigma-\check\Sigma_{x\phi}\Sigma_\phi^\dagger\check\Sigma_{x\phi}^\top\succ 0$, where $\check\Sigma_{x\phi}=\mathbb E(\check x_1^{(1)}\check\phi_1^{(1)\top})$ and $\Sigma_\phi$ is introduced in Assumption~\ref{assump:input_exciting}.
    \end{enumerate}
\end{assumption}
\vspace{-\baselineskip}%
\vspace{-\baselineskip}%
\begin{rese}
\begin{remark}
    Assumption~\ref{assump:nonGaussian} ensures that $\Xi^{(N)}$ has full row rank  almost surely and that the equality constraint in~\eqref{eq:alpha_equation} in our algorithm is therefore feasible.
\end{remark}
\end{rese}

Notably, under Assumption~\ref{assump:nonGaussian}, the KF serves as the optimal estimator among all linear estimators in the minimum mean square error sense, even for systems with non-Gaussian noise:
\begin{proposition}[\cite{kf_mmse_optimal}]\label{prop:nonGaussian_KF}
     \res{Consider the system~\eqref{eq:linear_system} under Assumption~\ref{assump:nonGaussian}. Let the KF recursion~\eqref{eq:kf} be initialized with $(\mu, \Sigma)$ and applied over $(u_{ini}, y_{ini}, u_f)$.  
    Then, the KF produces the MMSE
    prediction of $y_f$ among all linear estimators with mean
    \[
        \hat y_f
        = \eta_f(\Sigma)\, z_{T_{ini}}
        + O_f\, \hat\psi_{T_{ini}+1}^-(\Sigma)\, \mu,
    \]
    and covariance
    \[
        \hat{\mathcal Y}_f(\Sigma)
        = O_f P_{T_{ini}+1}^-(\Sigma) O_f^\top 
        + G_f Q_f G_f^\top 
        + R_f.
    \]}
    The definition of $\eta_f, \hat\psi_{T_{ini}+1}^-, P_{T_{ini}+1}^-, O_f, G_f, Q_f, R_f$ is the same as that in Proposition~\ref{prop:kf_distribution}.
\end{proposition}

The following proposition extends our convergence results in Theorem~\ref{thm:convergence_N_multiple} and Theorem~\ref{thm:offline_asymptotic_Tini} to systems with non-Gaussian noise:
\begin{proposition}[Extension to systems with non-Gaussian noise]\label{cor:kf_nonGaussian}
    Under Assumption~\ref{assump:persistently_exciting}--\ref{assump:nonGaussian}, \res{the following results hold:}
    \begin{enumerate}
        \item \res{The coefficients of the conditional mean and covariance of samples generated by CGM converge almost surely to their linear-MMSE counterparts}:
        \begin{align}
            &\lim_{N\to\infty}\frac{\|\res{\Theta_f^{(N)}}-\eta_{f}(\Sigma)\|}{N^{-0.5+\epsilon}}=0\ \textrm{a.s.}, \\
            &\lim_{N\to\infty}\frac{\|\res{\Sigma_f^{(N)}}-\mathcal{Y}_f(\Sigma)\|}{N^{-0.5+\epsilon}}=0\ \textrm{a.s.},
        \end{align}
        where $\Sigma$ shares the same expression as in Theorem~\ref{thm:convergence_N_multiple}.
        \item \res{The exponential convergence with respect to $T_{ini}$ established in
        Theorem~\ref{thm:offline_asymptotic_Tini} 
        remains valid in the non-Gaussian setting, with the same convergence rate.}
    \end{enumerate}
\end{proposition}
The proof of the proposition is included in Appendix~\ref{append:convergence_N}. This result demonstrates that our framework achieves the best possible linear estimation with correct mean/variance propagation asymptotically, even in non-Gaussian settings.

\section{Direct Data-Driven Predictive Control via Conditional Generative Modeling}\label{sec:control}
This section presents data-driven predictive control approaches based on the proposed CGM. \res{We develop the control formulations for systems with Gaussian noise, which enables a 
clear and tractable analysis.} First, we introduce a data-driven ``pessimistic" controller inspired by the scenario-based stochastic MPC. Subsequently, we introduce an alternative interpretation of DeePC~\cite{coulson2019data}. Building on this interpretation, we formulate an ``optimistic" controller that incorporates the proposed CGM to effectively capture the stochastic characteristic of the system.

\subsection{\res{Predictive Controller Inspired by Scenario-based MPC}}
Scenario-based Stochastic MPC (SSMPC)~\cite{schildbach_scenario_2014} is a powerful framework for solving finite-horizon constrained control problems for stochastic systems when complete parametric models are available. Specifically, for the stochastic system~\eqref{eq:linear_system} with Gaussian noise characterized by state-space parameters $A, B, C$ and noise covariances $Q, R$, the SSMPC solves the following control problem~\cite{li_stochastic_2023} at time $T_{ini}$ over horizon $T$ by generating $M$ noise scenarios:
\begin{problem}[Scenario-based stochastic MPC]\label{prob:scenario_based_mpc}
\begin{align}
    &\min_{u_f, y_f^{(j)}, x_{t}^{(j)}}\ \sum_{j=1}^MJ\left(y_f^{(j)}, u_f\right),\nonumber \\
    \mathrm{s.t.} \quad  &x_{t+1}^{(j)}=A x_{t}^{(j)}+B u_{t}+w_{t}^{(j)}, y_{t}^{(j)}=C x_{t}^{(j)}+v_{t}^{(j)}, \label{eq:state_space_model_in_SMPC} \\
    & x_{T_{ini}+1}^{(j)}=\bar x_{T_{ini}+1}^{-(j)},\nonumber \\
    &y_t^{(j)}\in\mathcal Y, u_t\in\mathcal U, \nonumber \\
    &\forall t=T_{ini}+1, \cdots, T_{ini}+T, j=1, \cdots, M, \label{eq:scenario_based_mpc}
\end{align}
where $u_f=\mathrm{col}(u_{T_{ini}+1}, \cdots, u_{T_{ini}+T}), y_f^{(j)}=\mathrm{col}(y_{T_{ini}+1}^{(j)}, \cdots, y_{T_{ini}+T}^{(j)})$, and $J$ denotes the cost function of the control problem. $w_t^{(j)}, v_t^{(j)}$ and $\hat x_{T_{ini}+1}^{-(j)}$ are \res{samples drawn from} $\mathcal N(0, Q), \mathcal N(0, R)$ and $\mathcal N(\hat x_{T_{ini}+1}^-, P_{T_{ini}+1})$ respectively, \res{and are treated as fixed scenarios in~\eqref{eq:scenario_based_mpc}. Here,} $\hat x_{T_{ini}+1}^-, P_{T_{ini}+1}^-$ are \res{obtained from} the corresponding Kalman filter of system~\eqref{eq:linear_system}. $\mathcal Y$ and $\mathcal U$ are the feasible sets of outputs and inputs, respectively.
\end{problem}

Notice that SSMPC requires full access to the state-space parameters $A, B, C$ and the noise covariance $Q, R$ of the system~\eqref{eq:linear_system}. In contrast, \res{we construct a \emph{direct data-driven} predictive controller by replacing the model-based scenario generation in~\eqref{eq:scenario_based_mpc} with samples generated by our CGM.
} Specifically, given \res{dataset $\mathcal{D}^{(N)}$ collected offline}, and $T_{ini}$ historical sample pairs $u_{ini}, y_{ini}$ of the current trajectory, we formulate the following optimization problem:
\begin{problem}[SSMPC inspired Direct Data-Driven Predictive Control]\label{prob:mpc_with_generative_model}
    \vspace{-0.1cm}
\begin{align}
    &\min_{u_f, y_f^{(j)}, \alpha^{(j)}}\ \sum_{j=1}^MJ\left(y_f^{(j)}, u_f\right),\nonumber \\
    \mathrm{s.t.}  \ &\alpha^{(j)}=(\arg\min_{g}\|g\|_2^2\mid \res{\Xi^{(N)}}g=\mathrm{col}(0, z_{T_{ini}}))\nonumber \\
    &+(\beta^{(j)}\mid \res{\Xi^{(N)}}\beta^{(j)}=0),y_f^{(j)}=\res{Y_f^{(N)}}\alpha^{(j)},\nonumber \\
    & z_{T_{ini}}=\mathrm{col}(u_{ini}, y_{ini}, u_f),\nonumber \\
    & y_{t}^{(j)}\in{\mathcal Y}, u_{t}\in{\mathcal U},\nonumber \\
    &t=T_{ini}+1, \cdots, T_{ini}+T, j=1, \cdots, M, \label{eq:nonparametric_mpc}
\end{align}
\res{where each $\beta^{(j)}$ is a sample from $\mathcal N(0, \frac{1}{N}I_N)$ and is treated as fixed in the optimization problem.}
\end{problem}

This problem is formulated in a behavioral manner, as it only involves the \res{trajectory} library \res{$\mathcal D^{(N)}$} and the initial trajectory $u_{ini}, y_{ini}$. Similar to SSMPC, this formulation is convex and can be efficiently solved using standard optimization tools, as long as the control cost \( J \) is convex and \( \mathcal{U}, \mathcal{Y} \) are convex sets. However, in contrast to SSMPC methods, which require the application of a Kalman filter and an optimization problem, both of which depend on precise system parameters, our approach is purely behavioral, avoiding an additional step for explicit system identification.

\subsection{Discussions: Relationship with DeePC}
This subsection provides an alternative interpretation of DeePC~\cite{coulson2019data} through the lens of the proposed CGM. DeePC is a recently developed direct data-driven control method that leverages the Willems' lemma~\cite{fundamental_lemma} to predict system behavior and formulate the following control problem:
\begin{problem}[DeePC for noisy systems]\label{prob:noisy_deepc}
\begin{align}
    \min_{u_f, y_f, g, \sigma_y, \sigma} &J\left(y_f, u_f\right)+\res{\gamma_g h_g(g)+\gamma_yh_y(\sigma_y)}, \nonumber\\
   \mathrm{ s.t. }\ &\res{Z^{(N)}}g=(z_{T_{ini}}+\sigma),\quad y_f=\res{Y_f^{(N)}}g, \label{eq:fundamental_lemma_deepc}\\
    &\sigma=\mathrm{col}(0, \sigma_y, 0),\quad z_{T_{ini}}=\mathrm{col}(u_{ini}, y_{ini}, u_f),\nonumber \\
    &y_t\in\mathcal Y,u_{t} \in \mathcal{U}, \quad \forall t=T_{ini}+1, \ldots, T_{ini}+T,\label{eq:deepc_noisy}
\end{align}
where $\gamma_g, \gamma_y$ are weighting parameters \res{$h_g, h_y$ denote penalty functions on $g$ and $\sigma_y$, respectively}.
\end{problem}

\begin{rese}
\vspace{-\baselineskip}%
\begin{remark}
    Various regularizers have been proposed in DeePC with different rationales: $\ell_1$-norms for distributional robustness~\cite{coulson2019data,coulson_distributionally_2021}, $\ell_2$-norms for computational simplicity and ridge-type interpretation~\cite{MATTSSON2023625}, and projection-based terms for consistency with system identification~\cite{dorfler_bridging_2023}.
\end{remark}
\end{rese}

Notice that in this problem,~\eqref{eq:fundamental_lemma_deepc} represents a behavior model for the noisy system\res{~\eqref{eq:linear_system}}, where \( g \) and \( \sigma_y \) are related to the noise contained in the \res{trajectory library $\mathcal D$} and the collected past outputs \( y_{ini} \), respectively. Based on this model, Problem~\ref{prob:noisy_deepc} seeks to jointly minimize two competing objectives: the control cost and the deviation of the optimization variables $y_f$ from the future outputs \res{determined by $\mathcal D^{(N)}$ and the past data of the current trajectory}.

Two observations regarding this formulation can be made: First, Problem~\ref{prob:noisy_deepc} adopts an optimistic stance toward noise, as the noise is treated as an optimization variable to minimize the control objective, instead of conforming to its correct distribution (stochastic formulation)~\cite{schildbach_scenario_2014} or serving as an adversarial element against minimizing the control objective (robust formulation)~\cite{4789462}. Secondly, for the Gaussian system~\eqref{eq:linear_system}, the distribution of the predicted $y_f$ by the model in~\eqref{eq:fundamental_lemma_deepc} may not effectively characterize the distribution of $y_f$ of the true system.

To alleviate the \res{model-mismatch} issue in the second observation, we replace the behavior model in Problem~\ref{prob:noisy_deepc} with our CGM \res{introduced in Section~\ref{sec:method}}.
Following the same logic as in Problem~\ref{prob:noisy_deepc}, \res{we treat the noise component $\beta$ in our CGM as an optimization variable, rather than a random variable.} Our objective function also jointly minimizes the control cost and the deviation of the optimization variable $y_f$ from the future outputs predicted by the CGM. The latter goal is achieved by maximizing the log-likelihood $\mathcal L(y_f\!\mid\! z_{T_{ini}}; \res{\mathcal{D}^{(N)}})$ of observing the generated $y_f$ given the initial trajectory $z_{T_{ini}}$. The resulting control problem is then formulated as:

\begin{align}
    \min_{u_f, y_f, \alpha, \beta}\ &J(u_f, y_f)-\gamma \mathcal L(y_f\mid z_{T_{ini}}; \res{\mathcal{D}^{(N)}}),\nonumber \\
    \mathrm{s.t.}\ 
    &\alpha=(\arg\min_{g}\|g\|_2^2\mid \res{\Xi^{(N)}}g=\mathrm{col}(0, z_{T_{ini}}))\nonumber \\
    &\qquad\qquad\qquad+(\beta\mid \res{\Xi^{(N)}}\beta=0),\nonumber \\
    & y_f=\res{Y_f^{(N)}}\alpha,\quad z_{T_{ini}}=\mathrm{col}(u_{ini}, y_{ini}, u_f),\nonumber \\
    & y_t\in\mathcal Y, u_{t} \in \mathcal{U}, \forall t=T_{ini}+1, \ldots, T_{ini}+T.\label{eq:generative_model_deepc}
\end{align}
\res{Similar to Problem~\ref{prob:noisy_deepc}, the hyperparameter $\gamma$ balances the control cost and the model consistency.} 

The following proposition derives an equivalent expression of $\mathcal L$ \res{for the Gaussian system~\eqref{eq:linear_system}}:
\begin{proposition}\label{prop:constraint_equivalent}
    The log-likelihood function $\mathcal L(y_f\mid z_{T_{ini}}; \mathcal{D}^N)$ has the following explicit expansion:
    \begin{align}
        \mathcal L(y_f\mid z_{T_{ini}}; \res{\mathcal{D}^{(N)}})&=-\frac{N}{2}\|\res{Y_f^{(N)}}(I-\res{\Xi^{(N)\dagger}\Xi^{(N)}})\beta\|_{\Gamma^{-1/2}}^2\nonumber \\
        &-\frac{Tp}{2}\log(2\pi)-\frac{1}{2}\log\det(\frac{1}{N}\Gamma),\label{eq:log_likelihood}
    \end{align}
    where $\Gamma=\res{Y_f^{(N)}(I-\Xi^{(N)\dagger}\Xi^{(N)})(I-\Xi^{(N)\dagger}\Xi^{(N)})^\top Y_f^{(N)\top}}$ and $\|x\|_{\Gamma^{-1/2}}=\sqrt{x^\top\Gamma^{-1} x}$.
\end{proposition}
\res{This proposition follows from the standard log-likelihood function for multivariate Gaussian distributions~\cite{bishop2006pattern}.} Using Proposition~\ref{prop:constraint_equivalent}, we can derive another stochastic variant of DeePC in~\eqref{eq:fundamental_lemma_deepc}:
\begin{problem}[Variant of DeePC based on our CGM]\label{prob:ours_optimistic}
\begin{align}
    &\min_{u_f, y_f, \alpha, \beta} J\left(y_f, u_f\right)+\tilde\gamma\|\res{Y_f^{(N)}}(I-\res{\Xi^{(N)\dagger}\Xi^{(N)}})\beta\|_{\Gamma^{-1/2}}^2, \nonumber\\
    \text { s.t. }\ &\alpha=(\arg\min_{g}\|g\|_2^2\mid \res{\Xi^{(N)}}g=z_{T_{ini}})\nonumber \\
    &\qquad\qquad\qquad\qquad\qquad+(\beta\mid \res{\Xi^{(N)}}\beta=0),\nonumber \\
    &y_f=\res{Y_f^{(N)}}\alpha, \quad z_{T_{ini}}=\mathrm{col}(u_{ini}, y_{ini}, u_f),\nonumber \\
    &y_t\in\mathcal Y,u_{t} \in \mathcal{U} \quad \forall t=T_{ini}+1, \ldots, T_{ini}+T.\label{eq:optimistic_ours}
\end{align}
\end{problem}

Similar to Problem~\ref{prob:noisy_deepc}, maximizing the likelihood function also leads to a regularization term, which better aligns with the stochastic properties of the system.

Finally, we note that by setting the random part $\beta$ in our CGM of Problem~\ref{prob:ours_optimistic} to zero, the controller reduces to a data-driven deterministic controller. Moreover, if \res{$\Phi^{(N)}=0$} in addition to $\beta=0$, the resulting controller coincides with the Subspace Predictive Controller (SPC)~\cite{FAVOREEL19994004}. Hence, this establishes SPC as a data-driven version of the deterministic MPC, where only the expectation of the predicted outputs are considered. Simulation results in Section~\ref{sec:simulations} compares its performance with both deterministic MPC and DeePC.

\section{Simulations}\label{sec:simulations}
This section presents simulations to assess the effectiveness of the proposed model. We focus on the trajectory tracking problem of a 4-dimensional single-input, single-output three-pulley system, originally introduced in~\cite{hjalmarsson1995model} and subsequently studied in~\cite{LANDAU199577, CAMPI20021337, MARKOVSKY202142, dorfler_bridging_2023}. The system parameters are consistent with those in~\cite{CAMPI20021337}. Both process and observation noise are modeled as Gaussian with zero mean and covariances of \( Q = 0.01I \) and \( R = 0.04I \), respectively.

\subsection{Convergence Speed Verification}
We perform experiments to evaluate the convergence speed in Theorem~\ref{thm:convergence_N_multiple}. For the CGM, we set the length of initial trajectory to $T_{ini} = 8$ and the prediction horizon to $T = 10$. Each experiment is repeated independently for $100$ times, and the results are presented in Fig.~\ref{fig:convergence}. The figures show that the convergence speed is approximately $O(N^{-0.5})$ in both cases, supporting the validity of Theorem~\ref{thm:convergence_N_multiple}.

\begin{figure}[!htbp]
    \centering
    \begin{subfigure}[t]{0.22\textwidth}
        \centering
        \includegraphics[width=\textwidth]{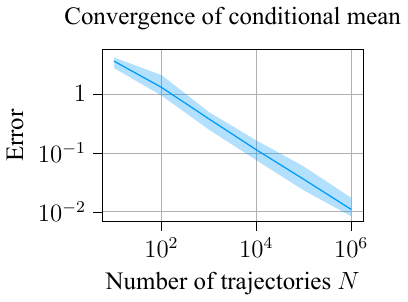}
        \caption{Convergence of the coefficients of the conditional mean considered in Theorem~\ref{thm:convergence_N_multiple} when the behavior library $\mathcal D$ consists of multiple independent trajectories.}
    \end{subfigure}
    \begin{subfigure}[t]{0.23\textwidth}
        \centering
        \includegraphics[width=\textwidth]{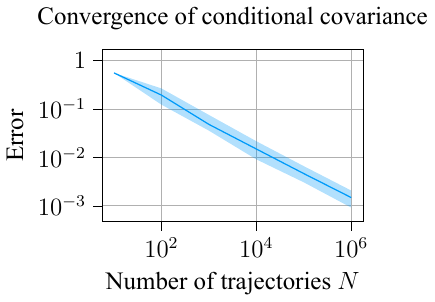}
        \caption{Convergence of the conditional covariance considered in Theorem~\ref{thm:convergence_N_multiple} when the behavior library $\mathcal D$ consists of multiple independent trajectories.}
    \end{subfigure}
    \begin{subfigure}[t]{0.23\textwidth}
        \centering
        \includegraphics[width=\textwidth]{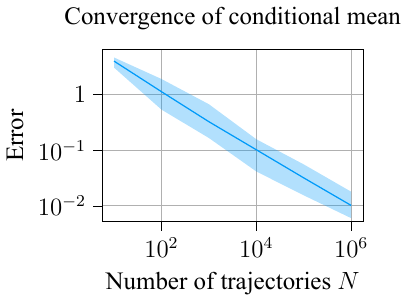}
        \caption{Convergence of coefficients of the conditional mean considered in Theorem~\ref{thm:convergence_N_multiple} when the behavior library $\mathcal D$ is a single trajectory.}
    \end{subfigure}
    \begin{subfigure}[t]{0.23\textwidth}
        \centering
        \includegraphics[width=\textwidth]{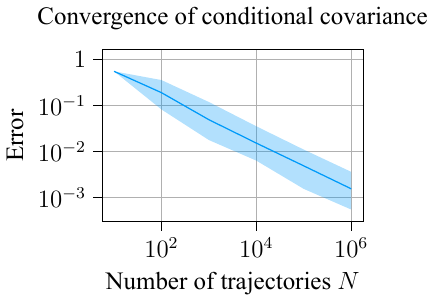}
        \caption{Convergence of the conditional covariance considered in Theorem~\ref{thm:convergence_N_multiple} when the behavior library $\mathcal D$ is a single trajectory.}
    \end{subfigure}
    \caption{Convergence speed verification of the proposed generative model as the number of trajectories $N$ (or equivalently, the length of the single trajectory) increases. The figures depict a log-log ribbon plot, with the solid line indicating the mean value and the shaded region representing the range of values.}
    \label{fig:convergence}
    \vspace{-0.6cm}
\end{figure}

\subsection{Control Performance Comparison}
We next compare the control performance of the direct data-driven predictive controllers proposed in Section~\ref{sec:control}.

\res{\textbf{Controller Implementations:}} We implement three groups of controllers:
\begin{itemize}
    \item \textbf{Deterministic controllers}: We implement several deterministic control approaches for comparison. First, we reproduce a classic model-based controller combining Kalman filtering with deterministic MPC (KF+DMPC), where the state-space parameters are set to the true system parameters. We also implement the Subspace Predictive Controller (SPC)~\cite{FAVOREEL19994004}, which can be viewed as a deterministic variant of our proposed optimistic controller in Problem~\ref{prob:ours_optimistic}. Additionally, we include a traditional data-driven control approach that first identifies state-space parameters using the Ho-Kalman algorithm, then applies Kalman filtering and deterministic MPC (Ho-Kalman+KF+DMPC)~\cite{oymak_revisiting_2022}.
    \item \textbf{``Optimistic'' controllers}: We implement DeePC, as described in Problem~\ref{prob:noisy_deepc}, \res{$\gamma$-DDPC~\cite{breschi2023data}} and our variant of DeePC from Problem~\ref{prob:ours_optimistic} with $\tilde\gamma=1000$. To enhance computational efficiency and improve the performance of DeePC, following~\cite{dorfler_bridging_2023, MATTSSON2023625}, we \res{specify} the regularization term in Problem~\ref{prob:noisy_deepc} \res{as}:
    \[
    \gamma_y\|\sigma_y\|_2^2 + \gamma_Z\|(I - \res{Z^{(N)\dagger} Z^{(N)}})g\|_2^2,
    \]
    where we set $\gamma_y = 100$ and $\gamma_Z = 10^6$.
    \item \textbf{Scenario-based stochastic controllers}: We implement a classic model-based controller, specifically a KF combined with an SSMPC (KF+SSMPC), as described in Problem~\ref{prob:scenario_based_mpc}, where the state-space parameters and noise covariances are set to the true system parameters. \res{We also implement an oracle controller (Oracle SSMPC with perfect $x_t$) that assumes perfect knowledge of both system parameters and current states $x_t$ at each time step.} Additionally, we construct our proposed scenario-based predictive controller equipped with our generative model, as described in Problem~\ref{prob:mpc_with_generative_model}, for comparison. This controller is denoted as Scenario-based Stochastic PC with generative model (SSPC$_{gen}$). In all controllers, we use $M = 50$ scenarios to account for system uncertainty.
\end{itemize}
\res{\textbf{Experimental Setup:} All experiments start from the zero initial state. }
\res{Following~\cite{dinkla_closed-loop_2024}, the controllers apply zero inputs for the first $T_{ini}=8$ time steps to collect the initial trajectory $(u_{ini}, y_{ini})$. 
After this initialization period, each controller operates in closed loop for $100$ time steps.} 
The control objective is similar to~\cite{coulson2019data}, defined as:
\begin{align}
    J(y_f, u_f)=\sum_{t=T_{ini}+1}^{T_{ini}+T}\left(\|y_t-10\|_2^2+0.01\|u_t\|_2^2\right).\label{eq:control_cost}
\end{align}
The constraints are $y_t\leq 12, t=T_{ini}+1, \cdots, T_{ini}+T$. \res{The horizon for all controllers is set to $T = 10$. For the data-driven controllers, we generate single long trajectories with i.i.d. Gaussian inputs, with trajectory lengths $K = 117$ and $K = 1017$, setting $T_{ini} = 8$ and $T = 10$. These trajectory lengths allow us to construct Hankel matrices with $N = 100$ and $N = 1000$ columns respectively, following the relationship $N = K - T_{ini} - T + 1$, to formulate the \res{trajectory} library $\mathcal{D}^{(N)}$. We perform $100$ independent experiments, reporting the average control cost across experiments with no constraint violations, along with the probability of constraint violation, defined as
\[
p_{\text{fail}} = \frac{N_{\text{fail}}}{N_{\text{trial}}},
\]
where $N_{\text{fail}}$ denotes the number of experiments with constraint violations, and $N_{\text{trial}}$ denotes the total number of experiments.  Constraint violations are monitored and recorded throughout the entire closed-loop operation.} The results are presented in Table~\ref{tab:simulations}.

\begin{table*}[!htbp]
    \centering
    \res{
    \begin{tabular}{l|ccc}
        \hline
        Controller & Average Cost & $p_{fail}$ & Online computation time per step (ms) \\
        \hline
        KF+Deterministic MPC (KF+DMPC) & 436.94 & 0.19   & 2.30 \\
        SPC ($N=100$) & 443.12 & 0.20 & 2.62 \\
        SPC ($N=1000$) & 438.25 & 0.20 & 2.58\\
        Ho-Kalman+KF+Deterministic MPC (Ho-Kalman+KF+DMPC, $N=100$) & 1087.97 & 0.56 & 2.30 \\
        Ho-Kalman+KF+Deterministic MPC (Ho-Kalman+KF+DMPC, $N=1000$) & 466.99 & 0.23 & 2.30 \\
        \hline
        DeePC ($N=100$) & 450.60 & 0.18 & 7.62 \\
        DeePC ($N=1000$) & 438.50 & 0.14 & 792.07 \\
        $\gamma$-DDPC ($N=100$) & 443.02 & 0.19 & 1.13 \\
		$\gamma$-DDPC ($N=1000$) & 438.25 & 0.20 & 1.92 \\
        Our variant of DeePC ($N=100, \tilde\gamma=1000$,  DeePC$_{var, \tilde\gamma}$) & 443.06 & 0.20 & 10.43 \\
        Our variant of DeePC ($N=1000$, $\tilde\gamma=1000$, DeePC$_{var, \tilde\gamma}$) & 438.41 & 0.14 & 865.17 \\
        \hline
        Oracle SSMPC with perfect $x_t$ & 432.26 & 0.01 & 9.69 \\
        KF+SSMPC (KF+SSMPC) & 439.88 & 0.00 & 10.99 \\
        Scenario-based stochastic PC with CGM ($N=100$, SSPC$_{gen}$) & 440.73 & 0.05 & 13.33 \\
        Scenario-based stochastic PC with CGM ($N=1000$, SSPC$_{gen}$) & 437.32 & 0.01 & 12.67 \\
        \hline
    \end{tabular}}
    \caption{Average control cost, constraint violation probability $p_{fail}$ and online computation times for the deterministic controllers, ``optimistic'' controllers and scenario-based stochastic controllers, respectively, over $100$ independent experiments.}
    \label{tab:simulations}
    \vspace{-0.6cm}
\end{table*}

Simulation results show that while KF+DMPC achieves \res{a competitive} average control cost of \res{$436.94$}, it incurs a relatively high constraint violation probability of \res{$0.19$}. When the size of the \res{trajectory} library is rather limited ($N=100$), SPC and Ho-Kalman+KF+DMPC have \res{violation probabilities of $0.20$ and $0.56$, respectively}. \res{The particularly poor performance of Ho-Kalman+KF+DMPC with $N=100$ (average cost of $1087.97$ and violation probability of $0.56$) can be attributed to the sensitivity of the Ho-Kalman algorithm to noise in limited data~\cite{chiuso_ill-conditioning_2004,li2022fundamental}, leading to inaccurate identification of system parameters $A$, $B$, $C$ and consequently poor control performance.} However, \res{both algorithms} demonstrate significant improvement as the \res{trajectory} library size \( N \) increases, \res{with SPC reducing its control cost to $438.25$ and Ho-Kalman+KF+DMPC achieving substantial improvements in both cost ($466.99$) and violation probability ($0.23$)} at \( N = 1000 \). Additionally, DeePC, \res{$\gamma$-DDPC} and our DeePC$_{var, \tilde\gamma}$ exhibit suboptimal performance in terms of constraint satisfaction, although increasing the \res{trajectory} library size \( N \) enhances their reliability. These results show that an optimistic approach to noise in controller design can compromise robustness.

In contrast, KF+SSMPC reduces the violation probability to $0.00$, \res{while maintaining a competitive control cost of $439.88$}, demonstrating the importance of incorporating stochastic models for ensuring constraint satisfaction, particularly in safety-critical applications. Additionally, our proposed SSPC$_{gen}$ outperforms both the deterministic and optimistic controllers, \res{achieving violation probabilities of $0.05$ and $0.01$ for $N=100$ and $N=1000$ respectively, while maintaining control costs of $440.73$ and $437.32$}. As the size of the \res{trajectory} library $N$ increases, SSPC$_{gen}$ approaches the performance of SSMPC, which relies on true system parameters. This result further validates the effectiveness of the proposed generative model in accounting for system uncertainties and enhancing the robustness of predictive controllers against system noise.

Table~\ref{tab:simulations} \res{also} compares the online computation time per step for all controllers. The "optimistic" approaches (DeePC and DeePC$_{var, \tilde\gamma}$) require significantly longer computation times, especially with large \res{trajectory} libraries\res{, with computation times increasing from $7.62$ms to $792.07$ms for DeePC and from $10.43$ms to $865.17$ms for DeePC$_{var, \tilde\gamma}$ as $N$ increases from $100$ to $1000$}. In contrast, our SSPC$_{gen}$ and SPC, which are data-driven counterparts of SSMPC and KF+DMPC respectively, achieve comparable computational efficiency to their model-based versions\res{, with SSPC$_{gen}$ and SPC requiring only $13.33$ms and $2.62$ms for $N=100$ respectively}. Moreover, their computation times remain stable with increasing library size, demonstrating their strong scalability.

\begin{rese}
    At online deployment, we apply the proposed scenario-based stochastic predictive controller with the CGM (SSPC$_{\mathrm{gen}}$) on an unstable system. The objective, constraint set, and horizons are the same as in the preceding simulations; only the state-space matrices $(A,B,C)$ differ as follows:
    \[A = \begin{bmatrix} 1.10 & 0.20 \\ 0.00 & 0.98\end{bmatrix},\quad
        B = \begin{bmatrix} 0.50 \\ 0.20\end{bmatrix},\quad
        C = \begin{bmatrix} 1.00 & 0.00 \end{bmatrix}.\] 
    
    We also use $M=50$ scenarios per step. Performance metrics are kept the same as the preceding subsection, and the experimental results are reported in Table~\ref{tab:unstable_simulations}.

\begin{table*}[!htbp]
    \centering
    \begin{tabular}{l|ccc}
        \hline
        Controller & Average Cost & $p_{fail}$ & Online computation time per step (ms) \\
        \hline
        Scenario-based stochastic PC with CGM ($N=100$, SSPC$_{gen}$) & 122.16 & 0.00 & 11.87 \\
        Scenario-based stochastic PC with CGM ($N=1000$, SSPC$_{gen}$) & 118.69 & 0.00 & 12.21 \\
        \hline
    \end{tabular}
    \caption{Average control cost, constraint violation probability $p_{fail}$ and online computation times for scenario-based stochastic controllers, over $100$ independent experiments for an open-loop unstable system.}
    \label{tab:unstable_simulations}
    \vspace{-0.6cm}
\end{table*}

The simulation results demonstrate that the proposed CGM-based SSPC can effectively handle open-loop unstable systems. 
With data collected under the pre-stabilized output-feedback excitation, the controller achieves stable and constraint-satisfying performance, maintaining tracking accuracy comparable to the stable-plant case. 
These results confirm that the proposed method seamlessly extends to unstable dynamics while preserving strong control performance and robustness.
\end{rese}
 
\section{Conclusion}\label{sec:conclusions}
This paper proposes a behavioral CGM for stochastic LTI systems. The model generates random samples from the \res{data-driven approximation of the} conditional distribution of future outputs given past inputs and outputs, ensuring consistency with the system's \res{trajectory} library. It is purely behavioral, utilizing only the I/O samples of the system without explicit \res{system identification}. We analyze the model's performance by proving that the distribution of generated samples converges as the size of the \res{trajectory} library increases. Additionally, we prove that the gap between the asymptotic distribution as $N$ tends to infinity and the distribution of the optimal CGM decays exponentially with the length $T_{ini}$ of the initial trajectory. The proposed model is then incorporated into a predictive control framework, formulating a direct data-driven predictive control problem that accounts for the system's stochastic characteristics while bypassing the identification steps and \res{state estimations}. Numerical results validate the derived bounds and demonstrate the model's effectiveness in enhancing the robustness of predictive controllers under system noise.




\appendices
\section{Proof of Proposition~\ref{prop:kf_distribution}}\label{append:kf_distribution}
\begin{proof}
   \res{Consider} the current recorded initial trajectory \(z_{T_{ini}}\) beginning at $t=1$ and evolving to the current time step $T_{ini}$. \res{The KF recursively computes state estimates according to~\eqref{eq:kf}, initialized with 
    \(\hat x_{1}^-=\mu\) and \(P_{1}^- = P\).} Then, the optimal estimation for $x_{T_{ini}+1}$ conditioned on $u_{ini}, y_{ini}$ is $\hat x_{T_{ini}+1}^-$ with covariance $P_{T_{ini}+1}^-$. 
    \begin{rese}

    From the Kalman filter recursion~\eqref{eq:kf}, we observe that $\hat x_{T_{ini}+1}^-$ depends linearly on the measurements $y_{ini}$, inputs $u_{ini}$, and the initial mean $\mu$. By the linearity of the Kalman filter, this relationship can be expressed as:
    \begin{align}
        \hat x_{T_{ini}+1}^-=\hat\theta_{T_{ini}+1}^-(P)\mathrm{col}(u_{ini}, y_{ini})+\hat\psi_{T_{ini}+1}^-(P)\mu,\label{eq:theta_psi_Tini_def}
    \end{align}
    where \(\hat\theta_{T_{ini}+1}^- : \mathbb{R}^{n \times n} \to \mathbb{R}^{n \times (T_{ini}(m+p))}\) and \(\hat\psi_{T_{ini}+1}^- : \mathbb{R}^{n \times n} \to \mathbb{R}^n\) are functions of the initial state covariance \(P\), whose explicit forms can be computed recursively from~\eqref{eq:kf} but are omitted for brevity. The expressions also depend on the state-space parameters \(A, B, C, Q, R\), which are omitted in the notation for simplicity. We similarly treat $P_{T_{ini}+1}^-: \mathbb R^{n\times n}\to \mathbb R^{n\times n}$ as a function of $P$, denoted as $P_{T_{ini}+1}^-(P)$.
    \end{rese}

    Then, based on the estimated distribution of the state $x_{T_{ini}+1}$, we can provide multi-step ahead predictions for future outputs $y_f$. Define $w_f = \mathrm{col}(w_{T_{ini}+1}, \dots, w_{T_{ini}+T}), v_f = \mathrm{col}(v_{T_{ini}+1}, \dots, v_{T_{ini}+T})$ and
    \begin{align}
        O_f=\begin{bmatrix} C^\top & (CA)^\top & \cdots & (CA^{T-1})^\top \end{bmatrix}^\top,\nonumber
    \end{align}
    \begin{align}
        G_f=\begin{bmatrix} 0 & 0 & \cdots  & 0 & 0 \\
        C & 0 & \cdots & 0 & 0 \\
        \vdots & \vdots & \ddots & \vdots & \vdots \\
        CA^{T-3} & CA^{T-4} & \cdots & 0 & 0 \\
        CA^{T-2} & CA^{T-3} & \cdots & C & 0 \end{bmatrix},\nonumber
    \end{align}
    \begin{align}
        H_f&=G_f(I_T\otimes B),Q_f=(I_T\otimes Q), R_f=(I_T\otimes R).\nonumber
    \end{align}
    \res{By unrolling the system dynamics~\eqref{eq:linear_system} over the prediction horizon $T$, the relationship between future outputs $y_f$, the state $x_{T_{ini}+1}$ and future inputs $u_f$ is given by:}
    \[ y_f = O_f x_{T_{ini}+1} + H_f u_f + G_f w_f + v_f.\]
    \res{Using this relationship, the mean of $y_f$ conditioned on $u_{ini}, y_{ini}, u_f$ is:}
    \begin{align}
        &\mathbb E(y_f\mid z_{T_{ini}}) = O_f \hat x_{T_{ini}+1}^- + H_f u_f\nonumber \\
        &=\underbrace{[O_f\hat\theta_{T_{ini}+1}^-(P)\ H_f]}_{\eta_f(P)}z_{T_{ini}}+O_f\hat\psi_{T_{ini}+1}^-(P)\mu,\nonumber
    \end{align}
    \res{Since $w_f, v_f$ are independent of $z_{T_{ini}}$, the conditional covariance is:}
    \begin{align}
        &\mathcal Y_f(P)\triangleq\mathrm{Cov}(y_f\mid z_{T_{ini}}) \nonumber \\
        &= O_f P_{T_{ini}+1}^-(P) O_f^\top + G_f Q_f G_f^\top + R_f.\nonumber \qedhere
    \end{align}
\end{proof}

\section{Proofs of Theorem~\ref{thm:convergence_N_multiple} and Corollary~\ref{cor:kf_nonGaussian}}\label{append:convergence_N}
Both the conditional mean and covariance of $\bar y_f$ depend fundamentally on the matrix
\begin{rese}
\begin{align}
    &\mathcal{W}^{(N)} \nonumber \\
    &= \frac{1}{N}\begin{bmatrix}
    \Phi^{(N)\top}\ Z^{(N)\top}\ Y_f^{(N)\top}\end{bmatrix}^\top\begin{bmatrix} \Phi^{(N)\top}\ Z^{(N)\top}\ Y_f^{(N)\top}\end{bmatrix}.\nonumber
\end{align}
\end{rese}
For convenience of notation, when \res{$\Phi^{(N)}, Z^{(N)}, Y_f^{(N)}$} carry tilde ($\tilde\ $) or check ($\check\ $), these modifiers extend to \res{$\mathcal W^{(N)}$} correspondingly. Based on this definition, we first analyze the convergence property of \res{$\mathcal{W}^{(N)}$} as $N$ tends to infinity, which subsequently leads to the convergence of the conditional mean and covariance in Theorem~\ref{thm:convergence_N_multiple}.

\subsection{Convergence of \res{$\mathcal{W}^{(N)}$}}
To derive the asymptotic distribution of \res{$\mathcal W^{(N)}$} for both behavior libraries \res{$\mathcal{\tilde{D}}^{(N)}$} and \res{$\mathcal{\check{D}}^{(N)}$}, we first introduce several additional notations. For the multi-trajectory library \res{$\mathcal{\check D}^{(N)}$} under Assumption~\ref{assump:multiple_trajectories}, the corresponding random vectors \( \{\check u_t^{(i)}, \check y_t^{(i)}, \check\phi_t^{(i)}\}_{t=1}^{T_{\text{ini}}+T} \) of the sampled trajectories are i.i.d. across different indices $i$. To formalize this statistical structure, we define a set of representative random vectors \( \{\check u_t, \check y_t, \check\phi_t\}_{t=1}^{T_{\text{ini}}+T} \) that preserve the identical joint distribution shared by each trajectory \( \{\check u_t^{(i)}, \check y_t^{(i)}, \check\phi_t^{(i)}\}_{t=1}^{T_{\text{ini}}+T} \). Define the composite random vector
\[\check\vartheta_t\triangleq \mathrm{col}(\check\phi_1, \check u_{1:T_{ini}}, \check y_{1:T_{ini}}, \check u_{T_{ini}+1:T_{ini}+T}, \check y_{T_{ini}+1:T_{ini}+T}),\]
with the corresponding covariance matrix:
\[\mathcal{\check W}=\mathbb E(\check\vartheta_t\check\vartheta_t^\top).\]
On the other hand, for the single-trajectory library \res{$\mathcal{\tilde D}^{(N)}$}, let $\tilde y_{t, f}$ and $\tilde z_{t, T_{ini}}$ denote the random vectors corresponding to the $t$-th column of \res{$\tilde Y_f^{(N)}$} and \res{$\tilde Z^{(N)}$}, respectively:
\begin{align}
    &\tilde y_{t, f}=\mathrm{col}(\tilde y_{t+T_{ini}}, \cdots, \tilde y_{t+T_{ini}+T-1}),\nonumber \\
    &\tilde z_{t, T_{ini}}=\mathrm{col}(\tilde u_{t:t+T_{ini}-1}, \tilde y_{t:t+T_{ini}-1}, \tilde u_{t+T_{ini}:t+T_{ini}+T-1}).\nonumber
\end{align}
Then, define the augmented vector 
\res{\[\tilde\vartheta_t\triangleq \begin{bmatrix} 
        \tilde\phi_{t}^\top &  \tilde z_{t, T_{ini}}^{\top} & \tilde y_{t, f}^{\top}
    \end{bmatrix}^\top\]} and the corresponding asymptotic covariance matrix
\begin{align}
    \mathcal{\tilde W}&=\mathbb{\bar E}_t(\tilde\vartheta_t\tilde\vartheta_t^\top)=\lim_{t\to\infty}\mathbb{E}(\tilde\vartheta_t\tilde\vartheta_t^\top).
\end{align}
We hereafter unify the \res{notation} of $\mathcal{\check W}$ and $\mathcal{\tilde W}$ as $\mathcal W$ for simplicity.

\begin{rese}
For simplicity of notation, we introduce the following convergence rate notation. For a sequence of random variables, vectors, or matrices $\{x_t\}_{t \geq 1}$, we write $x_t \sim \mathcal{C}(\beta)$ if for all $\epsilon > 0$,
$$\lim_{t\to\infty}\frac{\|x_t\|}{t^{\beta+\epsilon}}=0 \quad \text{almost surely},$$
where both $\|x_t\|/t^{\beta+\epsilon}$ and $0$ are random variables.
This notation indicates that $x_t$ converges to zero at a rate faster than any polynomial of order $\beta$. Note that $\mathcal{C}(\beta)$ denotes a class of stochastic processes with a certain convergence rate, not a probability distribution.

We first recall the following result from \cite{liu2020online}, which establishes convergence rates for sample covariances of temporally correlated data from stable linear systems.

\begin{lemma}[\cite{liu2020online}, Lemma 7]\label{lem:liu_convergence}
    Consider the autonomous LTI system~\eqref{eq:linear_system} with $B=0$ and stable matrix $A$ (i.e., all eigenvalues of $A$ lie strictly inside the unit circle). Then the sample covariance converges as
    \[\frac{1}{k}\sum_{t=1}^{k}y_ty_t^\top - \mathbb {\bar E}_t(y_ty_t^\top) \sim \mathcal{C}(-0.5).\]
\end{lemma}
\end{rese}

Then, we introduce the following result:
\begin{lemma}
    Given the \res{trajectory} library \res{$\mathcal{D}^{(N)}$}, we have that:
    \[\res{\mathcal{W}^{(N)}}-\mathcal{W}\sim \mathcal C(-0.5).\]
\end{lemma}

\begin{proof}
    For the multi-trajectory library \res{$\mathcal{\check D}^{(N)}$}, the convergence of \res{$\mathcal{\check W}^{(N)}$} to $\mathcal W$ follows directly from the strong law of large numbers and the law of iterated logarithm. 
    
    For the single-trajectory library \res{$\mathcal{\tilde D}^{(N)}$}, we analyze \res{$\mathcal{\tilde W}^{(N)}$} (the sample average of $\tilde\vartheta_t\tilde\vartheta_t^\top$) by constructing an augmented system with output $\tilde\vartheta_t$. Define the $T_{ini}+T$-length windowed vector:
    \[\res{\tilde\phi_{t, T_{ini}:T}}=\mathrm{col}(\tilde\phi_t, \cdots, \tilde\phi_{t+T_{ini}+T-1}),\]
    \[\res{\tilde u_{t, T_{ini}:T}}=\mathrm{col}(\tilde u_t, \cdots, \tilde u_{t+T_{ini}+T-1}),\]
    and \res{$\tilde y_{t, T_{ini}:T}$} is similarly defined. Then, the augmented state vector becomes:
    \begin{align}
        \res{\tilde\varphi_t=\mathrm{col}(\tilde\phi_{t, T_{ini}:T}, \tilde u_{t, T_{ini}:T}, \tilde x_{t+T_{ini}+T-1},\tilde y_{t, T_{ini}:T}).}
    \end{align}
    Define an augmented system with the following dynamics:
    \begin{align}
        &\tilde\varphi_{t+1}=\tilde A_\varphi\tilde\varphi_t+\tilde w_{\varphi, t}, \quad\tilde\vartheta_t=\mathrm{col}(\tilde\phi_t, \res{\tilde z_{t, T_{ini}}, \tilde y_{t, f}})=\tilde C_\varphi\tilde\varphi_{t}, \label{eq:enlarged_system_proof}
    \end{align}
    where 
    $\tilde A_\varphi$ can be derived based on the dynamics of the system and the controller, and is omitted here due to the space limit.
    The noise $\tilde w_{\varphi, t}$ is a structured vector containing zero entries and noise terms $\res{\omega_{t+T_{ini}+T-1}}, w_{t+T_{ini}+T-1}$ and their linear transformations, 
    and $\tilde C_\varphi$ is a selection matrix taking the corresponding elements from $\tilde\varphi_t$. 

    \res{To prove the convergence of $\mathcal{\tilde W}^{(N)}$, we note that the primary challenge arises from the fact that the outputs $\tilde\vartheta_t$ of this augmented system are correlated across time, violating standard i.i.d. assumptions required for classical convergence results. However, since the augmented system in~\eqref{eq:enlarged_system_proof} is stable and driven by i.i.d. noise $\tilde w_{\varphi, t}$, it satisfies the conditions of Lemma~\ref{lem:liu_convergence}. Therefore, by Lemma~\ref{lem:liu_convergence}, the sample covariance of the temporally correlated outputs $\tilde\vartheta_t$ still converges at the standard rate, and it follows that:}
    \[\res{\mathcal{\tilde W}^{(N)}}-\mathcal{\tilde W}\sim\mathcal C(-0.5).\qedhere\]
\end{proof}
\subsection{Convergence of mean and covariance}
Next, we are ready to consider the convergence of mean and covariance of $\bar y_f$. For convenience of expressing the asymptotic distribution, for the single-trajectory behavior library $\mathcal{\tilde D}^N$, let
\[\tilde z_{t, T_{ini}}^\perp=\tilde z_{t, T_{ini}}-\mathbb{\bar E}_t(\tilde z_{t, T_{ini}}\tilde\phi_t^\top)\mathbb{\bar E}_t(\tilde\phi_t\tilde\phi_t^\top)^{\res{\dagger}}\tilde\phi_t,\]
\[\tilde y_{t, f}^\perp=\tilde y_{t, f}-\mathbb{\bar E}_t(\tilde y_{t, f}\tilde\phi_t^\top)\mathbb{\bar E}_t(\tilde\phi_t\tilde\phi_t^\top)^{\res{\dagger}}\tilde\phi_t.\]
On the other hand, for the multi-trajectory library \res{$\mathcal{\check D}^{(N)}$}, define
\[\check z_{T_{ini}}^\perp=\check z_{T_{ini}}-\mathbb{E}(\check z_{T_{ini}}\check\phi_1^\top)\mathbb{E}(\check\phi_1\check\phi_1^\top)^{\dagger}\check\phi_1,\]
\[\check y_{f}^\perp=\check y_{f}-\mathbb{E}(\check y_{f}\phi_1^\top)\mathbb{E}(\check\phi_1\check\phi_1^\top)^{\dagger}\check\phi_t.\]
Using these notations, we have the following convergence result:
\begin{lemma}\label{lemma:mean_variance_convergence}
    The coefficients in the mean of $\bar y_f$ satisfies
    \begin{align}
        &\res{\tilde\Theta_f^{(N)}}-\tilde\eta_{f, \infty}\sim \mathcal C(-0.5),\quad\res{\check\Theta_f^{(N)}}-\check\eta_{f, \infty}\sim \mathcal C(-0.5), \nonumber
    \end{align}
    where
    \[\tilde \eta_{f, \infty}=\mathbb{\bar E}_t(\tilde y_{t, f}^\perp\tilde z_{t,T_{ini}}^{\perp\top})[\mathbb{\bar E}_t(\tilde z_{t,T_{ini}}^\perp\tilde z_{t,T_{ini}}^{\perp\top})]^{-1}.\]
    \[\check \eta_{f, \infty}=\mathbb{E}(\check y_{t}^\perp\check z_{T_{ini}}^{\perp\top})[\mathbb{E}_t(\check z_{T_{ini}}^\perp\check z_{T_{ini}}^{\perp\top})]^{-1}.\]
    On the other hand, the covariance of the generated sample $\bar y_f$ satisfies:
    \begin{align}
        &\res{\tilde\Sigma_f^{(N)}}-\mathcal{\tilde Y}_{f, \infty}\sim \mathcal C(-0.5),\quad\res{\check\Sigma_f^{(N)}}-\mathcal{\check Y}_{f, \infty}\sim \mathcal C(-0.5),\nonumber
    \end{align}
    where
    \begin{align}
        \mathcal{\tilde Y}_{f, \infty}&=\mathbb{\bar E}_t(\tilde y_{t, f}^\perp\tilde y_{t, f}^{\perp\top})\nonumber \\
        &-\mathbb{\bar E}_t(\tilde y_{t, f}^\perp\tilde z_{t,T_{ini}}^{\perp\top})\mathbb{\bar E}_t(\tilde z_{t,T_{ini}}^\perp\tilde z_{t,T_{ini}}^{\perp\top})^{-1}\mathbb{\bar E}(\tilde z_{t,T_{ini}}^\perp\tilde y_{t, f}^{\perp\top}),\nonumber \\
        \mathcal{\check Y}_{f, \infty}&=\mathbb{E}(\check y_{f}^\perp\check y_{f}^{\perp\top})\nonumber \\
        &-\mathbb{E}(\check y_{f}^\perp\check z_{T_{ini}}^{\perp\top})\mathbb{E}(\check z_{T_{ini}}^\perp\check z_{T_{ini}}^{\perp\top})^{-1}\mathbb{E}(\check z_{T_{ini}}^\perp\check y_{f}^{\perp\top}).
    \end{align}
\end{lemma}
\begin{proof}
    This proof applies to both behavior libraries, so we omit the tilde and check accents. The key idea is to express the mean \res{$\Theta_f^{(N)}$} and covariance \res{$\Sigma_f^{(N)}$} as functions of \res{$\mathcal{W}^{(N)}$}, then establish their convergence using Lemma 3.3 from~\cite{liu2020online}.
    
    First, we derive the expression for \res{$\Theta_f^{(N)}$}. Recall that $\res{\Theta_f^{(N)}}z_{T_{ini}}$ represents the special solution:
    \begin{align}
        \res{Y_f^{(N)}}(\arg\min_g\|g\|^2_2\mid \mathrm{col}(\res{\Phi^{(N)}, Z^{(N)}})g=\mathrm{col}(0, z_{T_{ini}})).\label{eq:tilde_theta_f_equal}
    \end{align}
    The constraint $\res{\Phi^{(N)}}g=0$ implies $g=(I-\res{\Phi^{(N)\dagger}\Phi^{(N)}})h$ for any $h\in\mathbb R^N$. \res{Substituting this into the optimization problem in~\eqref{eq:tilde_theta_f_equal}, we obtain the equivalent formulation:}
    \begin{align}
        &(\arg\min_h\|(I-\res{\Phi^{(N)\dagger}\Phi^{(N)}})h\|_2^2\mid\nonumber \\
        &\qquad\qquad\res{Z^{(N)}}(I-\res{\Phi^{(N)\dagger}\Phi^{(N)}})h=z_{T_{ini}}).\label{eq:min_h}
    \end{align}
    Let \res{$Z^{(N)\perp}=Z^{(N)}(I-\Phi^{(N)\dagger}\Phi^{(N)})$}. The equality constraint in~\eqref{eq:min_h} implies:
    \begin{align}
        h=\res{Z^{(N)\perp\dagger}}z_{T_{ini}}+(I-\res{Z^{(N)\perp\dagger}Z^{(N)\perp}})\tilde h\label{eq:h_expression}
    \end{align}
    for any $\tilde h\in\mathbb R^N$. 
    Then, the objective function becomes:
    \begin{align}
        &\|g\|_2^2=\|(I-\res{\Phi^{(N)\dagger}\Phi^{(N)}})h\|_2^2\nonumber \\
        &=\|(I-\res{\Phi^{(N)\dagger}\Phi^{(N)})Z^{(N)\perp\dagger}}z_{T_{ini}}\|_2^2 \nonumber \\
        &\qquad\qquad+\|(I-\res{\Phi^{(N)\dagger}\Phi^{(N)}})(I-\res{Z^{(N)\perp\dagger}Z^{(N)\perp}})\tilde h\|_2^2.\nonumber
    \end{align}
    The minimum is achieved at $\tilde h=0$, yielding:
    \begin{align}
        (\arg\min_{g}\|g\|_2^2\mid 
        \res{\Xi^{(N)}}g&=\mathrm{col}(0, z_{T_{ini}}))\nonumber \\
        &=(I-\res{\Phi^{(N)\dagger}\Phi^{(N)})Z^{(N)\perp\dagger}}z_{T_{ini}}.\nonumber
    \end{align}
    Define $\res{Y_f^{(N)\perp}}=\res{Y_f^{(N)}}(I-\res{\Phi^{(N)\dagger}\Phi^{(N)}})$. This leads to the simplified expression:
    \[\res{\Theta_f^{(N)}=Y_f^{(N)\perp}Z^{(N)\perp\dagger}}.\]

    Next, we establish the connection between $\res{\mathcal{W}^{(N)}}$ and \res{$Y_f^{(N)\perp} Z^{(N)\perp\dagger}$}. We prove in Appendix~\ref{appendix:xi_full_rank} that \res{$Z^{(N)\perp}$} has full row rank, which implies that:
    \begin{rese}
    \begin{align}
    &Z^{(N)\perp\dagger}=Z^{(N)\perp\top}(Z^{(N)\perp}Z^{(N)\perp\top})^{-1}, \nonumber \\
    &Z^{(N)\perp}Z^{(N)\perp\top}=Z^{(N)}Z^{(N)\top}-Z^{(N)}\Phi^{(N)\dagger}\Phi^{(N)}Z^{(N)\top},\nonumber \\
    &=Z^{(N)}Z^{(N)\top}-Z^{(N)}\Phi^{(N)\top}(\Phi^{(N)}\Phi^{(N)\top})^\dagger\Phi^{(N)}Z^{(N)\top},\nonumber
    \end{align}
    \end{rese}
    where the last equality can be verified by computing the pseudo-inverse of $\res{\Phi^{(N)}}$ using its singular value decomposition.

    Then, we aim to express \res{$Y_f^{(N)\perp}Z^{(N)\perp\dagger}$} as a differentiable function of \res{$\mathcal W^{(N)}$}. Consider a matrix $X\in\mathbb R^{(n_u+(T_{ini}+T)(p+m))\times (n_u+(T_{ini}+T)(p+m))}$ with the same dimension as $\mathcal{W}$. For simplicity, we split $X$ as follows:
    \begin{align}
        X=\res{\begin{bmatrix}
            X_{\phi\phi} & X_{\phi z} & X_{\phi y} \\
            X_{z\phi} & X_{zz} & X_{zy} \\
            X_{y\phi} & X_{yz} & X_{yy}
        \end{bmatrix}},\label{eq:X_split}
    \end{align}
    where $X_{\phi\phi}\in\mathbb R^{n_u\times n_u}, X_{yy}\in\mathbb R^{Tp\times Tp}, X_{zz}\in\mathbb R^{(T_{ini}(m+p)+Tm)\times (T_{ini}(m+p)+Tm)}$, and the rest of the submatrices are determined accordingly. Define the function for the matrix $X$:
    \begin{align}
        &\mathcal A_3(X)\nonumber \\
        &=(X_{yz}-X_{y\phi}(X_{\phi\phi})^\dagger X_{\phi z})(X_{zz}-X_{z\phi}(X_{\phi\phi})^\dagger X_{\phi z})^{-1}.\nonumber
    \end{align}
    It can be established that $\mathcal A_3$ is differentiable at $\mathcal{W}$, and that 
    \[\res{Y_f^{(N)\perp} Z^{(N)\perp\dagger}}=\mathcal A_3(\res{\mathcal{W}^{(N)}}),\quad \res{\check\eta_{f, \infty}}=\mathcal A_3(\mathcal W).\]
    \res{Further define $\mathcal{\tilde A}_3(X)=\mathcal {A}_3(X+\mathcal W)-\check\eta_{f, \infty}$, then $\mathcal{\tilde A}_3$ is differentiable at 0 and $\mathcal{\tilde A}_3(0)=0$. Using Lemma 3~3) in~\cite{liu2020online},
    \res{\[\mathcal{A}_3(\mathcal W^{(N)}-\mathcal W)=\mathcal A_3(\mathcal W^{(N)})-\check\eta_{f, \infty}\sim \mathcal C(-0.5).\]}
    The first equation in the lemma then follows after some calculations regarding $\mathcal A_3(\mathcal{W})$.}

    Next, we consider the simplification of $\res{\Sigma_f^{(N)}}$. It can be shown through matrix calculations that:
    \[I-\res{\Xi^{(N)\dagger}\Xi^{(N)}}=(I-\res{\Phi^{(N)\dagger}\Phi^{(N)}})(I-\res{Z^{(N)\perp\dagger}Z^{(N)\perp}}).\]
    Thus, we have the simplification:
    \begin{rese}
    \begin{align}
        \Sigma_f^{(N)}=\frac{1}{N}Y_f^{(N)\perp}&(I-Z^{(N)\perp\dagger}Z^{(N)\perp})\cdot \nonumber \\
        &(I-Z^{(N)\perp\dagger}Z^{(N)\perp})^\top Y_f^{(N)\perp\top}.\nonumber
    \end{align}
    \vspace{-\baselineskip}%
    \end{rese}
    
    \res{Based on this simplification, similar to the case of $\Theta_f^{(N)}$, an explicit relationship between $\Sigma_f^{(N)}$ and $\mathcal W^{(N)}$ can be established (explicit expression omitted for brevity).} We then define the function for the matrix $X$ as follows:
    \begin{align}
        &\mathcal A_4(X)=X_{yy}-X_{y\phi}(X_{\phi\phi})^{\dagger}X_{\phi y}\nonumber \\
        &-(X_{yz}-X_{y\phi}(X_{\phi\phi})^{\dagger}X_{\phi z})
        (X_{zz}-X_{z\phi}(X_{\phi\phi})^{\dagger}X_{\phi z})^{-1}\nonumber \\
        &\qquad\qquad\qquad\qquad\cdot(X_{zy}-X_{z\phi}(X_{\phi\phi})^{\dagger}X_{\phi y}).\nonumber
    \end{align}
    \res{It can be established that $\mathcal A_4$ is differentiable at $\mathcal{W}, \Sigma_f^{(N)}=\mathcal A_4(\mathcal{W}^{(N)})$ and $\mathcal{\check Y}_{f, \infty}=\mathcal A_4(\mathcal{W})$. Using Lemma~3~3) in~\cite{liu2020online}, the second equation in the lemma follows.}
\end{proof}

\subsection{Establishing the relationship between the limit and KF}
We now prove by induction that the derived limit corresponds to the recursion of a KF~\eqref{eq:kf}. For brevity, we present the proof only for the single-trajectory case, as the multi-trajectory case follows similarly \res{by replacing} $\mathbb{\bar E}_t$ with $\mathbb E$. To establish this connection, we first introduce several notations. For the trajectory under consideration and $t=2, 3, \cdots$, define:
\[z_{t, ini}^{-}=\mathrm{col}(u_1, \cdots, u_{t-1}, y_1, \cdots, y_{t-1}),\]
\[z_{t, ini}=\mathrm{col}(u_1, \cdots, u_{t-1}, y_1, \cdots, y_t),\]
and $z_{1, ini}=y_1$. Then, for the KF in~\eqref{eq:kf} initialized by $\hat x_1^-=0, P_1^-=\Sigma$, similar to~\eqref{eq:theta_psi_Tini_def}, the predicted expectation of the KF can be simplified as:
\[\hat x_t^-=\hat\theta_t^-(\Sigma)z_{t, ini}^{-},\quad\hat x_t=\hat\theta_t(\Sigma)z_{t, ini},\]
where both coefficients $\hat\theta_t^-:\mathbb R^{n\times n}\to \mathbb R^{n\times (t-1)(p+m)}$ and $\hat\theta_t:\mathbb R^{n\times n}\to\mathbb R^{n\times (t-1)(p+m)+p}$ are functions of $\Sigma$ similar to the definition of $\hat \theta_{T_{ini}+1}^-$. We also define the covariance matrices $P_t^-, P_t$ as functions of $\Sigma$ similar to the definition of $P_{T_{ini}}^-$ in the main body of the paper.

On the other hand, define the following terms regarding the \res{trajectory} library as:
\begin{align}
    \tilde z_{t+l}^{-}=\mathrm{col}(\tilde u_{t+l-1}, \cdots, &\tilde u_{t}, \tilde y_{t}, \cdots, \tilde y_{t+l-1}),\nonumber \\
    \tilde z_{t+l}=\mathrm{col}(\tilde u_{t+l-1}, \cdots, &\tilde u_{t}, \tilde y_{t}, \cdots, \tilde y_{t+l}),\nonumber
\end{align}
with $l=1, \cdots, T_{ini}$ and $\tilde z_t=\tilde y_t$. Note that the ordering of the vectors $u$ in this definition is slightly different from the main body of the paper, which necessitates the use of a permutation matrix later. Moreover, in the rest of this subsection, for a random vector $h_{t+l}, l=0, 1, \cdots$, define the operation
\[h_{t+l}^\perp=h_{t+l}-\mathbb{\bar E}_t(h_{t+l}\tilde\phi_{t}^\top)\mathbb{\bar E}_t(\tilde\phi_{t}\tilde\phi_{t}^\top)^{\dagger}\tilde\phi_{t}.\]

Next, define the coefficients $\tilde\theta_{l+1}^-, \tilde\theta_l$ as follows for $l=1, \cdots, T_{ini}$:
\[\tilde\theta_l=\mathbb{\bar E}_t(\tilde x_{t+l-1}^\perp\tilde z_{t+l-1}^{\perp\top})\mathbb{\bar E}_t(\tilde z_{t+l-1}^\perp\tilde z_{t+l-1}^{\perp\top})^{-1},\]
\[\tilde\theta_{l+1}^-=\mathbb{\bar E}_t(\tilde x_{t+l}^\perp\tilde z_{t+l}^{-\perp\top})\mathbb{\bar E}_t(\tilde z_{t+l}^{-\perp}\tilde z_{t+l}^{-\perp\top})^{-1},\]
Similarly, define the covariance matrices:
\begin{align}
    \tilde P_l&=\mathbb{\bar E}_t(\tilde x_{t+l-1}^{\perp}\tilde x_{t+l-1}^{\perp\top})-\mathbb{\bar E}_t(\tilde x_{t+l-1}^\perp\tilde z_{t+l-1}^{\perp\top})\nonumber \\
    &\qquad\qquad\cdot\mathbb{\bar E}_t(\tilde z_{t+l-1}^{\perp}\tilde z_{t+l-1}^{\perp\top})^{-1}\mathbb{\bar E}_t(\tilde z_{t+l-1}^\perp\tilde x_{t+l-1}^{\perp\top}).\nonumber \\
    \tilde P_{l+1}^-&=\mathbb{\bar E}_t(\tilde x_{t+l}^{\perp}\tilde x_{t+l}^{\perp\top})\nonumber \\
    &-\mathbb{\bar E}_t(\tilde x_{t+l}^\perp\tilde z_{t+l}^{-\perp\top})\mathbb{\bar E}_t(\tilde z_{t+l}^{-\perp}\tilde z_{t+l}^{-\perp\top})^{-1}\mathbb{\bar E}_t(\tilde z_{t+l}^{-\perp}\tilde x_{t+l}^{\perp\top}),\nonumber
\end{align}

Next, we prove by induction that:
\[\tilde\theta_l=\hat\theta_{l}(\Sigma)\mathcal Q_{l-1, l},\quad \tilde P_l=P_l(\Sigma),\]
with $l=1, \cdots, T_{ini}$, where $\mathcal Q_{l, q}$ is a permutation matrix:
\[\mathcal Q_{l, q}=\begin{bmatrix} \mathcal Q_l & 0 \\ 0 & I_{pq}\end{bmatrix},\]
and $Q_l$ is a permutation matrix:
\[ \mathcal Q_l = [q_{ij}]_{l\times l} \otimes I_m,\ q_{ij}=\begin{cases} 1 & \text{if } i+j \equiv l+1 \\ 0 & \text{otherwise}\end{cases}.\]
This permutation arises from the different orderings of \( u \) in the definitions of \( \tilde{z}_{t+l} \) in this proof and in the main body.

\textbf{Base case ($l=1$)}: For $\Sigma=\mathbb{\bar E}_t(\tilde x_{t}^\perp\tilde x_{t}^{\perp\top})$, \res{it follows that}
\begin{align}
    \tilde P_{1}=\Sigma-\Sigma C^\top(C\Sigma C^\top+R)^{-1}C\Sigma.\nonumber
\end{align}
Moreover, according to the definition of $\tilde x_{t}^\perp$,
\begin{align}
    \Sigma&=\mathbb{\bar E}_t(\tilde x_{t}^\perp\tilde x_{t}^{\perp\top})=\mathbb{\bar E}_t(\tilde x_{t}\tilde x_{t}^\top)\nonumber \\
    &\qquad\qquad-\mathbb{\bar E}_t(\tilde x_{t}\tilde\phi_{t}^\top)\mathbb{\bar E}_t(\tilde\phi_{t}\tilde\phi_{t}^\top)^{\dagger}\mathbb{\bar E}_t(\tilde\phi_{t}\tilde x_{t}^\top).
\end{align}
Then, using the definition of $\mathbb{\bar E}_t$, \res{we obtain that} $\Sigma$ matches the expression in Theorem~\ref{thm:convergence_N_multiple}.

Similarly, the expression of $\tilde\theta_{1}$ is:
\[\tilde\theta_1=\Sigma C^\top(C\Sigma C^\top+R)^{-1}.\]
Hence, $\tilde\theta_1$ and $\tilde P_1$ correspond to the first measurement update step of the KF in~\eqref{eq:kf}.

\textbf{Induction step}: Assume the relations hold for some $l\geq 1$, we now prove that they hold for $l+1$ using the following lemmas:
\begin{lemma}\label{lemma:kf_prediction_recursion}
    The terms $\tilde P_l, \tilde P_{l+1}^-$ and $\tilde\theta_l, \tilde\theta_{l+1}^-$ have the following relationships respectively:
   \[\tilde P_{l+1}^-=A\tilde P_lA^\top+Q,\quad \tilde\theta_{l+1}^-=\begin{bmatrix}B & A\tilde\theta_{l}\end{bmatrix}.\]
\end{lemma}
\begin{proof}
    We first consider the relationship between $\tilde P_{l+1}^-$ and $\tilde P_l$. For simplicity of \res{notation}, let
    \[\Gamma_{xx}=\mathbb{\bar E}_t(\tilde x_{t+l-1}^\perp\tilde x_{t+l-1}^{\perp\top}), \Gamma_{xz}=\Gamma_{zx}^\top=\mathbb{\bar E}_t(\tilde x_{t+l-1}^\perp\tilde z_{t+l-1}^{\perp\top}),\]
    \[ \Gamma_{zz}=\mathbb{\bar E}_t(\tilde z_{t+l-1}^\perp\tilde z_{t+l-1}^{\perp\top}), \Gamma_{xu}=\Gamma_{xu}^\top=\mathbb{\bar E}_t(\tilde x_{t+l-1}^\perp\tilde u_{t+l-1}^\perp),\]
    \[ \Gamma_{uu}=\mathbb{\bar E}_t(\tilde u_{t+l-1}^\perp\tilde u_{t+l-1}^\perp).\]
    Based on these definitions, $\tilde P_l$ and $\tilde\theta_l$ can be expressed as:
    \[\tilde P_l=\Gamma_{xx}-\Gamma_{xz}\Gamma_{zz}^{-1}\Gamma_{zx},\quad\tilde\theta_l=\Gamma_{xz}\Gamma_{zz}^{-1},\]
    and
    \begin{align}
        &\mathbb{\bar E}_t(\tilde x_{t+l}^\perp\tilde x_{t+l}^{\perp\top})=A\Gamma_{xx}A^\top+A\Gamma_{xu}B^\top \nonumber \\
        &\qquad\qquad\qquad+B\Gamma_{ux}A^\top+B\Gamma_{uu}B^\top+Q.\nonumber
    \end{align}
    Moreover, based on the definition of $\tilde z_{t+l-1}^\perp$ and $\tilde z_{t+l}^{-\perp }$: 
    \[\tilde z_{t+l}^{-\perp }=\mathrm{col}(\tilde u_{t+l-1}^\perp,\tilde z_{t+l-1}^\perp).\]
    Hence,
    \begin{align}
        \mathbb{\bar E}(\tilde x_{t+l}^{\perp}\tilde z_{t+l}^{-\perp})=[A\Gamma_{xz}+B\Gamma_{uz}\ A\Gamma_{xu}+B\Gamma_{uu}],\label{eq:x_z_minus_recursion}
    \end{align}
    \[\mathbb{\bar E}(\tilde z_{t+l}^{-\perp}\tilde z_{t+l}^{-\perp \top})=\begin{bmatrix} \Gamma_{uu} & \Gamma_{uz} \\ \Gamma_{zu} & \Gamma_{zz}\end{bmatrix}.\]
    Next, using block matrix inversion lemma to $\mathbb{\bar E}(\tilde z_{t+l}^{-\perp}\tilde z_{t+l}^{-\perp \top})$
    \res{and after algebraic manipulation, this yields}
    \begin{align}
        \tilde P_{l+1}^-=A\tilde P_{l}A^\top+Q+f(\Sigma_{xu}-\Sigma_{xz}\Sigma_{zz}^{-1}\Sigma_{zu}),\label{eq:P_l_plus_1_simplify}
    \end{align}
    where $f(\Sigma_{xu}-\Sigma_{xz}\Sigma_{zz}^{-1}\Sigma_{zu})$ is a linear function of $\Sigma_{xu}-\Sigma_{xz}\Sigma_{zz}^{-1}\Sigma_{zu}$.

    Notice that $\Sigma_{xz}\Sigma_{zz}^{-1}=\tilde\theta_l$. Then, the term can be expanded as:
    \begin{align}
        &\Sigma_{xu}-\Sigma_{xz}\Sigma_{zz}^{-1}\Sigma_{zu}=\mathbb {\bar E}_t((\tilde x_{t+l-1}^\perp-\tilde\theta_l\tilde z_{t+l-1}^\perp)\tilde u_{t+l-1}^{\perp\top}).\nonumber
    \end{align}
    According to the definition of $\tilde\theta_l$, \res{the following equation holds:}
    \[\mathbb{\bar E}_t((\tilde x_{t+l-1}^\perp-\tilde \theta_l\tilde z_{t+l-1}^\perp)\tilde z_{t+l-1}^{\perp\top})=0.\]
    Moreover, according to the \res{dynamics of the controller in Assumption~\ref{assump:input_exciting}}, \res{it follows that} $\tilde u_{t+l-1}=\vartheta_1\phi_{t}+\vartheta_2\tilde z_{t+l-1}+\res{\tilde\nu_{t+l-1}}$, where $\vartheta_1$ and $\vartheta_2$ are \res{appropriately dimensioned matrices that depend on the controller parameters}. Then,
    \begin{align}
        &\tilde u_{t+l-1}^\perp=\tilde u_{t+l-1}-\mathbb{\bar E}_t(\tilde u_{t+l-1}\tilde\phi_{t}^\top)\mathbb{\bar E}_t(\tilde\phi_{t}\tilde\phi_{t}^\top)^{\dagger}\tilde\phi_{t}\nonumber \\
        &=\vartheta_2\tilde z_{t+l-1}^\perp+\res{\tilde \nu_{t+l-1}}.\nonumber
    \end{align}
    Hence,
    \begin{align}
        &\Sigma_{xu}-\Sigma_{xz}\Sigma_{zz}^{-1}\Sigma_{zu}\nonumber \\
        &=\mathbb{\bar E}_t((\tilde x_{t+l-1}^\perp-\tilde\theta_l\tilde z_{t+l-1}^\perp)\tilde z_{t+l-1}^{\perp\top})\vartheta_2^\top\nonumber \\
        &\res{+\mathbb{\bar E}_t ((\tilde x_{t+l-1}^{\perp}-\tilde\theta_l\tilde z_{t+l-1}^\perp)\tilde \nu_{t+l-1}^\top)}=0.\label{eq:sigma_x_u_zero}
    \end{align}
    Then, combining the result above with~\eqref{eq:P_l_plus_1_simplify}, \res{we obtain}
    \[\tilde P_{l+1}^-=A\tilde P_{l}A^\top+Q.\]

    \res{Moreover, combining the recursion in~\eqref{eq:x_z_minus_recursion} with equation~\eqref{eq:sigma_x_u_zero} yields}
    \begin{align}
        \mathbb{\bar E}_t(\tilde x_{t+l}^\perp\tilde z_{t+l}^{-\perp\top})\mathbb{\bar E}_t(\tilde z_{t+l}^{-\perp}\tilde z_{t+l}^{-\perp\top})&=\begin{bmatrix}B & A\Sigma_{xz}\Sigma_{zz}^{-1}\end{bmatrix}\nonumber \\
        &=\begin{bmatrix}B & A\tilde \theta_l\end{bmatrix}.\nonumber\qedhere
    \end{align}
\end{proof}

Hence, Lemma~\ref{lemma:kf_prediction_recursion} shows that the terms $\tilde P_l$, $\tilde P_{l+1}^-$, and $\tilde\theta_l\mathcal Q_{l-1, l}$, $\tilde\theta_{l+1}^-\mathcal Q_{l, l}$ align with a single time update step of the KF in~\eqref{eq:kf}.

Next, we introduce another lemma to address the measurement update step of the KF:
\begin{lemma}\label{lemma:kf_measurement_recursion}
    Let $\tilde K_l=\tilde P_lC^\top(C\tilde P_lC^\top+R)^{-1}$, then, the terms $\tilde P_l^-, \tilde P_{l}$ and $\tilde\theta_l^-$ and $\tilde\theta_l$ are related as follows:
    \[\tilde P_{l}=(I-\tilde K_lC)\tilde P_{l}^-,\quad\tilde\theta_{l+1}^-=\begin{bmatrix}(I-\tilde K_lC)\tilde\theta_l^- & \tilde K_l\end{bmatrix}.\]
\end{lemma}
\begin{proof}
    We first consider the relationship between $\tilde P_{l}^-$ and $\tilde P_l$. For simplicity of \res{notation}, let
    \[\Gamma_{xx}=\mathbb{\bar E}_t(\tilde x_{t+l-1}^\perp\tilde x_{t+l-1}^{\perp\top}), \Gamma_{xz}=\Gamma_{zx}^\top=\mathbb{\bar E}_t(\tilde x_{t+l-1}^\perp\tilde z_{t+l-1}^{-\perp\top}),\]
    \[\Gamma_{zz}=\mathbb{\bar E}_t(\tilde z_{t+l-1}^{-\perp}\tilde z_{t+l-1}^{-\perp\top}), \Gamma_{xu}=\Gamma_{ux}^\top=\mathbb{\bar E}_t(\tilde x_{t+l-1}^\perp\tilde u_{t+l-1}^{\perp\top}),\]
    \[\Gamma_{uu}=\mathbb{\bar E}_t(\tilde u_{t+l-1}^\perp\tilde u_{t+l-1}^{\perp\top}).\]
    Then,
    \[\tilde P_l^-=\Gamma_{xx}-\Gamma_{xz}\Gamma_{zz}^{-1}\Gamma_{zx},\quad \tilde\theta_l^-=\Gamma_{xz}\Gamma_{zz}^{-1}.\]

    Notice that $\tilde z_{t+l-1}^\perp=\mathrm{col}(\tilde z_{t+l-1}^{-\perp}, \tilde y_{t+l-1}^\perp)$, \res{it follows that}
    \[\mathbb{\bar E}_t(\tilde x_{t+l-1}^\perp\tilde z_{t+l-1}^{\perp\top})=\begin{bmatrix}\Gamma_{xz} & \Gamma_{xx}C^\top\end{bmatrix},\]
    \[\mathbb{\bar E}_t(\tilde z_{t+l-1}^{\perp}\tilde z_{t+l-1}^{\perp\top})=\begin{bmatrix}\Gamma_{zz} & \Gamma_{zx}C^\top \\ C\Gamma_{xz} & C\Gamma_{xx}C^\top+R\end{bmatrix}.\]
    Then, using block matrix inversion lemma to $\mathbb{\bar E}_t(\tilde z_{t+l-1}^{\perp}\tilde z_{t+l-1}^{\perp\top})$ and \res{by routine calculations, we arrive at}
    \[\tilde P_{l}=\tilde P_l^--\tilde P_{l}^-C^\top FC\tilde P_{l}^-,\]
    where $F=(C\Gamma_{xx}C^\top+R)^{-1}$.
    Similarly, \res{it follows that}
    \[\tilde\theta_{l}=\begin{bmatrix}\tilde\theta_{l}^--\tilde P_{l}^-C^\top F\tilde\theta_{l}^- & \tilde P_{l}^-C^\top F\end{bmatrix},\]
    Let $\tilde K_{l}=\tilde P_{l}^-C^\top F=\tilde P_{l}^-C^\top (C\tilde P_{l}^-C^\top+R)^{-1}$. \res{Substituting this expression yields the desired identity and completes the proof.}
\end{proof}
From Lemma~\ref{lemma:kf_measurement_recursion}, \res{it can be deduced that} the relationships between $\tilde P_l$, $\tilde P_l^-$, and $\tilde \theta_l\mathcal Q_{l-1, l}$, $\tilde \theta_l^-\mathcal Q_{l-1, l-1}$ correspond to the measurement update step of the Kalman Filter.

Hence, by mathematical induction, we have established that:
\[\tilde\theta_l=\hat\theta_{l}(\Sigma)\mathcal Q_{l-1, l},\quad \tilde P_l=P_l(\Sigma),l=1, \cdots, T_{ini}.\]
By applying Lemma~\ref{lemma:kf_prediction_recursion}, \res{it can also be shown that:}
\[\tilde\theta_l^-=\hat\theta_{l}^-(\Sigma)\mathcal Q_{l-1, l-1}, \tilde P_l^-=P_l^-(\Sigma),l=2, \cdots, T_{ini}+1.\]
Finally, let us examine the relationship between $\tilde P_{T_{ini}+1}^-$ and the asymptotic covariance $\mathcal{\tilde Y}_{f, \infty}$ in Lemma~\ref{lemma:mean_variance_convergence}, as well as $\tilde\theta_{T_{ini}+1}^-$ and $\tilde\eta_{f, \infty}$. Define:
\[\tilde z_{t+T_{ini}+T}^{\perp-}=\mathrm{col}(\mathcal Q_{T_{ini}-1, T_{ini}}\tilde z_{t+T_{ini}}^\perp, \tilde u_f^\perp),\]
where $\tilde u_f=\mathrm{col}(\tilde u_{t+T_{ini}}, \cdots, \tilde u_{t+T_{ini}+T-1})$.
Additionally, \res{it follows that:}
\[\tilde y_f^\perp=O_f\tilde x_{t+T_{ini}+1}^\perp+H_f\tilde u_f^\perp+G_f\tilde w_f+\tilde v_f,\]
    where $\tilde w_f=\mathrm{col}(\tilde w_{t+T_{ini}-1}, \cdots, \tilde w_{t+T_{ini}+T-1}), \tilde v_f=\mathrm{col}(\tilde v_{t+T_{ini}}, \cdots, \tilde v_{t+T_{ini}+T})$. Using a similar technique as in the proof of Lemma~\ref{lemma:kf_prediction_recursion}, \res{we obtain:}
\begin{align}
    \mathcal{\tilde Y}_{f, \infty}&=O_f\tilde P_{T_{ini}+1}^-O_f^\top+G_fQ_fG_f^\top+R_f \nonumber \\
    &=O_fP_{T_{ini}+1}^-(\Sigma)O_f^\top+G_fQ_fG_f^\top+R_f,\nonumber \\
    \tilde\eta_{f, \infty}&=[O_f\mathcal Q_{T_{ini}-1, T_{ini}}\tilde\theta_{T_{ini}+1}^- H_f]=[O_f\hat\theta_{T_{ini}+1}^-(\Sigma)\ H_f].\nonumber
\end{align}

Thus, the statements of Theorem~\ref{thm:convergence_N_multiple} have been fully proven. Moreover, since the proof of Theorem~\ref{thm:convergence_N_multiple} and the proof of Theorem~\ref{thm:offline_asymptotic_Tini} in Appendix~\ref{append:offline_asymptotic_Tini} does not rely on any Gaussian assumptions, Proposition~\ref{cor:kf_nonGaussian} can also be \res{established} using the same proofs.

\section{Proof of Theorem~\ref{thm:offline_asymptotic_Tini}}\label{append:offline_asymptotic_Tini}
\begin{proof}
    The gap between the two distributions considered in Theorem~\ref{thm:offline_asymptotic_Tini} is equivalent to the gap between the following two distributions:

\begin{itemize}
    \item The estimation of \( x_{T_{ini}+1} \) using \( T_{ini} \) initial samples, as discussed in Theorem~\ref{thm:convergence_N_multiple}. This estimation is the output of a KF that takes the past \( T_{ini} \) samples as inputs, initialized with \( (0, \Sigma) \).
    \item The estimation of \( x_{T_{ini}+1} \) using all available past samples, as described in Proposition~\ref{prop:optimal_distribution}. This estimation is the output of a KF that takes all past samples as inputs, initialized with \( (\mu_{\mathcal T}, \Sigma_{\mathcal T}) \).
\end{itemize}

For both KFs, we consider the time interval from \( t = 1 \) to \( t = T_{ini} \). During this interval, both filters share the same state-space parameters and input-output samples, with the only difference being the initial distribution of \( x_1 \), \res{which is} \( (0, \Sigma) \) for the first filter and \( (\hat{x}_1^-, P_1^-) \) for the second.

Hence, as \( T_{ini} \to \infty \), applying Lemma~\ref{lemma:kf_initial_state} proves the theorem.
\end{proof}

\begin{lemma}[Stability of KF with respect to its initialization]\label{lemma:kf_initial_state}
    For the LTI system~\eqref{eq:linear_system}, suppose $(A, C)$ is detectable and $(A, Q^{1/2})$ is stabilizable. Let $\hat x_{t, 11}^-, P_{t, 1}^-$ and $\hat x_{t, 22}^-, P_{t, 2}^-$ denote the output of the Kalman filter recursively updated by the same I/O samples, but initialized by two different conditions $(\mu_1, \Sigma_1)$ and $(\mu_2, \Sigma_2)$ as in~\eqref{eq:kf}. Then, there exists $0<\rho<1$, such that
    \[\lim_{t\to\infty}\|P_{t, 1}^--P_{t, 2}^-\|\rho^{-t}=0,\quad \lim_{t\to\infty}\|\hat x_{t, 11}^--\hat x_{t, 22}^-\|\rho^{-t}=0\ a.s.\]
\end{lemma}
\begin{proof}
   First, due to the assumptions in the lemma, we know from~\cite[Chapter 4]{anderson2005optimal},~\cite{kf_initial_condition} \res{that for any initialization of the covariance in the Kalman filter, there exists a constant $\Phi>0$, such that}
    \[\|P_{t}^--P^*\|\leq \Phi\underline\rho^t,\]
    where
    \(\underline\rho=\rho(A-AK^*C)<1,\)
    \res{$\rho(X)$ denotes the spectral radius of the matrix $X$} and $P^*$ is the unique positive definite solution to the Riccati equation:
    \[P^*=A P^*A^\top-A P^*C^\top(C P^*C^\top+R)^{-1}CP^*A^\top +Q.\]
    Hence, we can conclude that for any $\underline\rho< \rho<1$,
    \[\lim_{t\to\infty}\|P_{t, 1}^--P^*\|\rho^{-t}=0, \lim_{t\to\infty}\|P_{t, 2}^--P^*\|\rho^{-t}=0.\]
    Therefore,
    \begin{align}\label{eq:P_diff_convergence}
        &0\leq\limsup_{t\to\infty}\|P_{t, 1}^--P_{t, 2}^-\|\rho^{-t} \nonumber \\
        &\leq \lim_{t\to\infty}\|P_{t, 1}^--P^*\|\rho^{-t}+\lim_{t\to\infty}\|P_{t, 2}^--P^*\|\rho^{-t}=0.
    \end{align}
    We next consider the second equation in the lemma. Let the corresponding Kalman gain of $P_{t, 1}^-$ and $P_{t, 2}^-$ be $K_{t, 1}$ and $K_{t, 2}$ respectively. Since
    \[K_{t, *}=P_{t, *}^- C^\top(CP_{t, *}^-C^\top+R)^{-1},\]
    where $*$ can be either $1$ or $2$, we bound the difference between $K_{t, 1}$ and $K_{t, 2}$ as:
    \begin{align}
        &K_{t, 1}-K_{t, 2} \nonumber \\
        &=P_{t, 1}^-C^\top(CP_{t, 1}^-C^\top+R)^{-1}-P_{t, 2}^-C^\top(CP_{t, 1}^-C^\top+R)^{-1}\nonumber \\
        &+P_{t, 2}^-C^\top(CP_{t, 1}^-C^\top+R)^{-1}-P_{t, 2}^-C^\top(CP_{t, 2}^-C^\top+R)^{-1}.\label{eq:K_diff}
    \end{align}
    The difference between the first two terms in the equation above are bounded by the difference between $P_{t, 1}^-$ and $P_{t, 2}^-$. It remains to bound the difference between the last two terms. \res{
    Let 
    $\displaystyle S_{1,t}=CP_{t,1}^-C^\top+R$, 
    $\displaystyle S_{2,t}=CP_{t,2}^-C^\top+R$, 
    and 
    $\displaystyle E_t=C(P_{t,1}^- - P_{t,2}^-)C^\top$.
    Then
    \[
    \|S_{1,t}^{-1}-S_{2,t}^{-1}\|
    \le \|S_{1,t}^{-1}\|\,\|S_{2,t}^{-1}\|\|E_t\|.
    \]
    Since $R\succ0$, both $S_{1,t}$ and $S_{2,t}$ are uniformly positive definite, and thus
    \(\|S_{1,t}^{-1}\|\|S_{2,t}^{-1}\|\le \tilde\Phi\)
    for some constant $\tilde\Phi>0$.
    Hence, }combining the boundedness above with~\eqref{eq:P_diff_convergence} and~\eqref{eq:K_diff}, for all $\underline\rho<\rho<1$,
    \begin{equation}\label{eq:K_12_converge}
        \lim_{t\to\infty}\|K_{t, 1}-K_{t, 2}\|\rho^{-t}=0.
    \end{equation}
    
    Now, we are ready to consider the difference between $\hat x_{t, 11}^-$ and $\hat x_{t, 22}^-$, which can be \res{derived} as:
    \begin{align}
        &\hat x_{t, 11}^--\hat x_{t, 22}^-=(A-AK_{t-1, 1}C)\hat x_{t-1, 11}^-\nonumber \\
        &+AK_{t-1, 1}y_{t-1}-(A-AK_{t-1, 2}C)\hat x_{t-1, 22}^--AK_{t-1, 2}y_{t-1}.\nonumber
    \end{align}
    Let $\chi_t=\hat x_{t, 11}^--\hat x_{t, 22}^-$ and let $\tilde K_t=K_{t, 1}-K_{t, 2}$. \res{It follows that}
    \begin{align}
        &\chi_t=\hat x_{t, 11}^--\hat x_{t, 22}^-\nonumber \\
        &=\underbrace{(A-AK_{t-1, 1}C)}_{\tilde A_{t-1}}\chi_{t-1}+A\tilde K_{t-1}\underbrace{(y_{t-1}-C\hat x_{t-1, 22}^-)}_{\epsilon_{t-1}}.\label{eq:chi_t_recursion}
    \end{align}
    For a sufficiently large $t$, consider the following common Lyapunov function:
    \begin{align}
        V(\chi_t)=\chi_t^\top P^*\chi_t.\nonumber
    \end{align}
    Then, using the recursion in~\eqref{eq:chi_t_recursion},
    \begin{align}
        V(\chi_{t+1})=\chi_{t}^\top\tilde A_{t}^\top P^*\tilde A_{t}\chi_{t}+\Gamma_{t+1},\nonumber
    \end{align}
    where
    \begin{align}
        &\Gamma_{t+1}=\chi_{t}^\top\tilde A_{t}^\top P^*A\tilde K_{t}\epsilon_{t}+\epsilon_{t}^\top \tilde K_{t}^\top A^\top P^*\tilde A_{t}\chi_{t}\nonumber \\
        &\qquad\qquad\qquad+\epsilon_{t}^\top\tilde K_{t}^\top A^\top P^* A\tilde K_{t}\epsilon_{t}.\label{eq:Gamma_def}
    \end{align}
    Hence, let $\tilde A=(A-AK^*C)$,
    \begin{align}
        &V(\chi_{t+1})-V(\chi_{t})=\chi_{t}^\top(\tilde A^\top P^*\tilde A-P^*)\chi_{t}\nonumber \\
        &\qquad\qquad+\chi_{t}^\top(\tilde A_{t}^\top P^*\tilde A_{t}-\tilde A^\top P^*\tilde A)\chi_{t}+\Gamma_{t+1}.\nonumber
    \end{align}
    Notice that $P^*$ satisfies:
    \[P^*=\tilde A^\top P^*\tilde A+A^\top K^{*\top}RK^*A+Q,\]
    \res{it can be deduced that}
    \begin{align}
        &V(\chi_{t+1})-V(\chi_{t})=-\chi_{t}^\top(A^\top K^{*\top}RK^*A+Q)\chi_{t}\nonumber \\
        &+\chi_{t}^\top(\tilde A_{t}^\top P^*\tilde A_{t}-\tilde A^\top P^*\tilde A)\chi_{t}+\Gamma_{t+1}.\label{eq:V_diff}
    \end{align}
    According to~\eqref{eq:K_12_converge}, there exists constants $L_1, L_2, L_3>0$, such that
    \[\|\tilde A_{t}^\top P^*\tilde A_{t}-\tilde A^\top P^*\tilde A\|\leq L_1\rho^t,\]
    \[|\Gamma_{t+1}|\leq L_2\|\chi_{t}\|\rho^t+L_3\rho^{2t}\ a.s.,\]
    for all $\underline\rho<\rho<1$. Let $\eta=\lambda_{\min}(A^\top K^{*\top}RK^*A+Q)>0$. Then, for a fixed $\rho\in(\underline\rho, 1)$, there exists constants $L>0, \underline\eta>0$ and a time index $T_1>0$, such that for all $t\geq T_1$,
    \[\eta-L_1\rho^t-\frac{L_2L+L_3}{L^2}\geq \underline\eta>0.\]
    Suppose there exits $t\geq T_1$, such that $\|\chi_{t}\|\geq L\rho^t$. Otherwise, the proposition is already proved. Then,
    \begin{align}
        &V(\chi_{t+1})-V(\chi_{t})\leq -\|\chi_{t}\|^2(\eta-L_1\rho^t)+|\Gamma_{t+1}|\nonumber \\
        &\leq -\|\chi_{t}\|^2(\eta-L_1\rho^t)+(\frac{L_2L+L_3}{L^2})\|\chi_{t}\|^2\nonumber \\
        &\leq -\underline\eta\|\chi_{t}\|^2\leq -\frac{\underline\eta}{\|P^*\|}V(\chi_{t})\ a.s.
    \end{align}
    Moreover, we can verify through the definition of $\underline\eta$ that $0<\underline\eta/\|P^*\|<1$. Hence, we can prove the exponentially convergence of $\|\chi_t\|$ to $0$ for all $t\geq T_1$, as long as $\|\chi_t\|\geq L\rho^t$. Then, we can conclude that there exists constants $\rho<1, L_4>0$, such that for all $t\geq T_1$,
    \begin{align}
        \|\chi_t\|\leq L_4\rho^t.\label{eq:chi_t_bound}
    \end{align}
    Hence, we can prove the second equation in the lemma.
\end{proof}

\begin{rese}
\section{Proof of Remark~\ref{rem:xi_full_rank}}\label{appendix:xi_full_rank}
\begin{theorem}\label{thm:Xi_full_row_rank}
Suppose the assumptions in 
Section~\ref{sec:problem_formulation} hold, and let 
\[
r \;\triangleq\; n_{\mathrm{ctrl}} + m(T_{\mathrm{ini}}+T) + pT_{\mathrm{ini}}
\]
denote the number of rows of $\Xi^{(N)}$.  
If $N\geq r(T_{ini}+T+\kappa_{\mathrm{ctrl}})+1$, then $\Xi^{(N)}$ has full row rank almost surely, i.e.,
\[
\mathbb{P}\bigl( \mathrm{rank}\,\Xi^{(N)} = r \bigr) = 1,
\]
whenever the noise process satisfies one of the following conditions:
\begin{enumerate}
    \item[\textnormal{(i)}] the Gaussian noise model in Section~\ref{sec:problem_formulation}; 
    \item[\textnormal{(ii)}] the non-Gaussian model in Assumption~\ref{assump:nonGaussian}.
\end{enumerate}
\end{theorem}

\begin{proof}
	We first consider the case where the system noise is Gaussian and the trajectory library is constructed from a single trajectory. Recall that we assume that $v_t\sim\mathcal N(0,R)$ and $\nu_t\sim\mathcal N(0,R_{\mathrm{ctrl}})$ with $R\succ0$, $R_{\mathrm{ctrl}}\succ0$, mutually independent and independent of initial conditions and plant process noise.
	The controller $(A_{\mathrm{ctrl}},B_{\mathrm{ctrl}},C_{\mathrm{ctrl}})$ is a minimal realization;
	let $n_{\mathrm{ctrl}}$ be its state dimension and $\kappa_{\mathrm{ctrl}}$ the controllability index of $(A_{\mathrm{ctrl}},B_{\mathrm{ctrl}})$.
	Recall that $K=T_{ini}+T+N-1$ denotes the length of the single trajectory in $\mathcal D$. 
		
	In this proof, we aim to show that $\Xi^{(N)}$ has full row rank almost surely. 
	Since its row dimension is $r$, it suffices to show that one can extract $r$ linearly independent columns. 
	Our strategy is as follows: 
	Part~I develops a universal triangular factorization for fixed \((u,y)\)-windows.
	Part~II extends this factorization by propagating the noise terms through the controller, thereby obtaining a decomposition of the combined vector formed by \((\phi, u, y)\), which constitutes columns of $\Xi^{(N)}$.
	Part~III selects \(r\) columns of $\Xi^{(N)}$ and concatenates them into a global block-lower-triangular map to establish that the selected submatrix is nonsingular.
	Finally, Part~IV concludes almost-sure full row rank of \(\Xi^{(N)}\) via an algebraic-variety and absolute-continuity argument.
		
	Define
	\[
	\tilde x_t=\mathrm{col}(x_t,y_t,\phi_t,u_t),\qquad
	\tilde y_t=\mathrm{col}(u_t,y_t),
	\]
	and the augmented system
	\begin{align*}
	\tilde x_{t+1}&=\tilde A\tilde x_t+\tilde w_t+\tilde G\tilde v_t,\\
	\tilde y_t&=\tilde C\tilde x_t,
	\end{align*}
	where
	\[
	\tilde w_t=\begin{bmatrix}w_t\\ Cw_t\\ 0\\ 0\end{bmatrix},\quad
	\tilde v_t=\begin{bmatrix}\nu_t\\ v_t\end{bmatrix},\quad
	\tilde C=\begin{bmatrix}0&0&0&I\\ 0&I&0&0\end{bmatrix},
	\]
	\[
	\tilde G=\begin{bmatrix}0&0\\ 0&I\\ 0&0\\ I&0\end{bmatrix},\ \tilde A=\begin{bmatrix}
	A & 0 & 0 & B \\
	CA & 0 & 0 & CB \\
	0 & B_{\mathrm{ctrl}} & A_{\mathrm{ctrl}} & 0 \\
	0 & C_{\mathrm{ctrl}}B_{\mathrm{ctrl}} & C_{\mathrm{ctrl}}A_{\mathrm{ctrl}} & 0
	\end{bmatrix}.
	\]
	A direct calculation gives $\tilde C\tilde G=\mathrm{diag}(I,I)$.

	\medskip
	\textbf{Part I: Triangular factorization for \((u,y)\)-windows.}
	Stack $\tilde y_{1:K}^d$:
	\[
	\tilde y_{1:K}^d
	=\tilde O_K\tilde x_0^d
	+ \tilde G_K\tilde w_{0:K-1}^d
	+ \tilde G_K(I_K\otimes\tilde G)\tilde v_{0:K-1}^d,
	\]
	where
	\[
	\tilde O_K=\begin{bmatrix}\tilde C\tilde A\\ \tilde C\tilde A^2\\ \vdots\\ \tilde C\tilde A^{K}\end{bmatrix},
	\tilde G_K=\begin{bmatrix}
	\tilde C & 0 & \cdots & 0 \\
	\tilde C\tilde A & \tilde C & \cdots & 0 \\
	\vdots & \vdots & \ddots & \vdots \\
	\tilde C\tilde A^{K-1} & \tilde C\tilde A^{K-2} & \cdots & \tilde C
	\end{bmatrix}.
	\]
	Since $\tilde C\tilde G=I$, the matrix $\tilde G_K(I_K\otimes\tilde G)$ is block \emph{lower} triangular with identity diagonal blocks and hence invertible.

	Let $\tilde{\mathcal Q}=\begin{bmatrix}I&0\end{bmatrix}$ so that $\tilde{\mathcal Q}\tilde y_t^d=u_t^d$.
	Define the fixed selection (reordering) operator
	\[
	\tilde{\mathcal Q}_{-T:-1}
	=\begin{bmatrix}
	I & 0\\[2pt]
	0 & (I_T\otimes \tilde{\mathcal Q})
	\end{bmatrix},
	\]
    \[\zeta\;\triangleq\;
	\mathrm{col}\!\bigl(\tilde y_{1:T_{\mathrm{ini}}+N-1}^d,\ u_{T_{\mathrm{ini}}+N:K}^d\bigr).\]
	Then
    \begin{align}
	\zeta
	=\tilde{\mathcal Q}_{-T:-1}\tilde y_{1:K}^d
	=&\underbrace{\tilde{\mathcal Q}_{-T:-1}\tilde O_K\tilde x_0^d
	+\tilde{\mathcal Q}_{-T:-1}\tilde G_K\tilde w_{0:K-1}^d}_{\boldsymbol{\rho}_\zeta}
	\nonumber \\
    &+\underbrace{\tilde{\mathcal Q}_{-T:-1}\tilde G_K(I_K\otimes\tilde G)}_{\mathsf T_\zeta}\tilde v_{0:K-1}^d.\nonumber
    \end{align}
	Here $\mathsf T_\zeta$ is obtained from a block lower-triangular, unit-diagonal matrix by deleting matched rows/columns; thus it remains block lower triangular with unit diagonal and is invertible. Since $\boldsymbol{\rho}_\zeta$ is independent of $\tilde v_{0:K-1}^d$ and
	$\mathrm{Cov}(\tilde v_{0:K-1}^d)=\mathrm{diag}(I\otimes R_{\mathrm{ctrl}},I\otimes R)\succ0$,
	\[
	\mathrm{Cov}(\zeta)
	=\mathrm{Cov}(\boldsymbol{\rho}_\zeta)+\mathsf T_\zeta\,\mathrm{Cov}(\tilde v_{0:K-1}^d)\,\mathsf T_\zeta^\top
	\ \succ\ 0.
	\]
	Using this proof, any $(u,y)$-stack extracted from $\tilde y_{1:K}^d$ by a fixed selection admits a representation “independent term $+$ invertible lower–triangular map $\times$ Gaussian noise” and is nondegenerate Gaussian.

	\medskip
	\textbf{Part II: Controller propagation and single-column decomposition of \(\Xi^{(N)}\).}
	Based on the factorization in Part~I, we now exploit the factorization of a specific column of $\Xi^{(N)}$.

	(i) The first column with $\phi_1$. Consider the first column
	$\mathbf c^{(1)}=\mathrm{col}(\phi_1^d,u_{1:T_{\mathrm{ini}}+T}^d,y_{1:T_{\mathrm{ini}}}^d)$.
	Define
	\[
	\boldsymbol{\eta}^{(1)}=\mathrm{col}(\phi_1^d,\ \nu_{1:T_{\mathrm{ini}}+T}^d,\ v_{1:T_{\mathrm{ini}}}^d),
	\]
	and write
	\[
	\mathbf c^{(1)}=\boldsymbol{\rho}^{(1)}+\mathsf T^{(1)}\boldsymbol{\eta}^{(1)},
	\qquad
	\mathsf T^{(1)}=
	\begin{bmatrix}
	I_{n_{\mathrm{ctrl}}} & 0 & 0 \\
	* & \mathsf L_u^{(1)} & 0 \\
	* & * & \mathsf L_y^{(1)}
	\end{bmatrix}.
	\]
	Here $\boldsymbol{\rho}^{(1)}$ is independent of $\boldsymbol{\eta}^{(1)}$; the blocks
	$\mathsf L_u^{(1)}$, $\mathsf L_y^{(1)}$ are unit-diagonal lower-triangular (from Part~I).
	Since $\mathrm{Cov}(\phi_1)\succ0$ and $R_{\mathrm{ctrl}},R\succ0$,
	$\mathrm{Cov}(\boldsymbol{\eta}^{(1)})\succ0$ and $\mathsf T^{(1)}$ has full row rank,
	so $\mathbf c^{(1)}$ is nondegenerate Gaussian.

	(ii) \emph{Noise-only factorization of the $y$-window and its injection into $\phi$.}
	Fix $s\ge0$.
	By the triangular factorization in Part~I, any fixed $y$-window admits
    \begin{align}
	    &y_{s+1:s+\kappa_{\mathrm{ctrl}}}\nonumber \\
	    &=
	    \underbrace{\rho_y^{(s)}}_{\text{independent of }v_{s+1:s+\kappa_{\mathrm{ctrl}}}}
	    +
	    \underbrace{T_y^{(s)}}_{\substack{\text{block lower triangular}\\ \text{unit diagonal}}}
	    v_{s+1:s+\kappa_{\mathrm{ctrl}}}.\nonumber
	\end{align}
	Unrolling the controller over $\kappa_{\mathrm{ctrl}}$ steps and injecting the above decomposition yields
    \begin{align}
	    &\phi_{s+\kappa_{\mathrm{ctrl}}+1}\nonumber \\
	    &=
	    \underbrace{A_{\mathrm{ctrl}}^{\kappa_{\mathrm{ctrl}}}\phi_{s+1}
	    + \mathcal C_{\mathrm{ctrl}}\,\rho_y^{(s)} + \zeta_s}_{\displaystyle \rho_\phi^{(s)}\ \perp\ v_{s+1:s+\kappa_{\mathrm{ctrl}}}}
	    +
	    \underbrace{\mathcal C_{\mathrm{ctrl}}T_y^{(s)}}_{\displaystyle \widetilde T_\phi^{(s)}}
	    v_{s+1:s+\kappa_{\mathrm{ctrl}}},\nonumber
	\end{align}
	where $\mathcal C_{\mathrm{ctrl}}\in\mathbb R^{n_{\mathrm{ctrl}}\times (p\kappa_{\mathrm{ctrl}})}$ denotes the $\kappa_{\mathrm{ctrl}}$-step controllability matrix of the stabilizing controller:
	\[
	\mathcal C_{\mathrm{ctrl}}
	=
	\begin{bmatrix}
	A_{\mathrm{ctrl}}^{\kappa_{\mathrm{ctrl}}-1}B_{\mathrm{ctrl}} &
	A_{\mathrm{ctrl}}^{\kappa_{\mathrm{ctrl}}-2}B_{\mathrm{ctrl}} & \cdots &
	B_{\mathrm{ctrl}}
	\end{bmatrix},
	\]
	$\widetilde T_\phi^{(s)}$ is a fixed block lower–triangular map (and has full row rank $n_{\mathrm{ctrl}}$ by controllability of $(A_{\mathrm{ctrl}},B_{\mathrm{ctrl}})$), 
	$\rho_\phi^{(s)}$ is independent of $v_{s+1:s+\kappa_{\mathrm{ctrl}}}$, and $\zeta_s$ collects state and noise terms up to time $s$. 
		
	Since $(A_{\mathrm{ctrl}},B_{\mathrm{ctrl}})$ is controllable and $\kappa_{\mathrm{ctrl}}$ is its controllability index, this block row has full row rank $n_{\mathrm{ctrl}}$. 
	Moreover, $T_y^{(s)}$ is invertible; hence $\widetilde{T}_\phi^{(s)}$ also has full row rank. 
	Therefore, there exists an index set $\mathcal J\subset\{1,\dots,p\kappa_{\mathrm{ctrl}}\}$ with $|\mathcal J|=n_{\mathrm{ctrl}}$ such that the square submatrix
	\[
	\widehat{\mathcal C}^{(s)}\;\triangleq\;(\widetilde{T}_\phi^{(s)})_{:,\,\mathcal J}\in\mathbb R^{n_{\mathrm{ctrl}}\times n_{\mathrm{ctrl}}}
	\]
	is nonsingular.
	As a result, the $s$-column $\mathbf c^{(s)}$ of $\Xi^{(N)}$ admits the following decomposition:
	\begin{align}
		\mathbf{c}^{(s)}=\boldsymbol{\rho}^{(s)}+\mathsf T^{(s)}\boldsymbol{\eta}^{(s)},
	\end{align}
	where 
	\[\boldsymbol{\rho}^{(s)}=\mathrm{col}(\rho_\phi^{(s)}, \rho_y^{(s)}),\]
    \[\boldsymbol{\eta}^{(s)}=\mathrm{col}(v_{s-\kappa_{\mathrm{ctrl}}:s-1}, \nu_{s:s+T_{ini}+T-1}, v_{s:s+T_{ini}-1}),\]
	with the block-lower-triangular structure
	\[
	\mathsf T^{(s)}=
	\begin{bmatrix}
	\widehat{\mathcal C}^{(s)} & 0 & 0 \\
	* & \mathsf L_u^{(s)} & 0 \\
	* & * & \mathsf L_y^{(s)}
	\end{bmatrix},
	\]
	where $\widehat{\mathcal C}^{(s)}$ is invertible,
	and $\mathsf L_u^{(s)},\mathsf L_y^{(s)}$ are unit-diagonal lower-triangular. Hence, $\mathsf T^{(s)}$ is invertible.

	\medskip
	\textbf{Part III: Column selection and block-lower-triangular concatenation.}
	Next, we are ready to introduce how to construct the $r$ columns to be picked. We set the start times at
	\[
	t_1\triangleq 1,\qquad
	t_k\triangleq k\bigl(T_{\mathrm{ini}}+T+\kappa_{\mathrm{ctrl}}\bigr)+1,\quad k=2,\dots,r,
	\]
	so that each picked column uses the blocks
	$\phi_{t_k}$, $u_{t_k:t_k+T_{\mathrm{ini}}+T-1}$, $y_{t_k:t_k+T_{\mathrm{ini}}-1}$.

	We now stack the $r$ selected columns together:
	\[
	\boldsymbol{\eta}_{\mathrm{all}}=\mathrm{col}\bigl(\boldsymbol{\eta}^{(1)},\dots,\boldsymbol{\eta}^{(r)}\bigr),\qquad
	\boldsymbol{\rho}_{\mathrm{all}}=\mathrm{col}\bigl(\boldsymbol{\rho}^{(1)},\dots,\boldsymbol{\rho}^{(r)}\bigr),
	\]
	and the column stack
	\[
	\mathbf C=\mathrm{col}\bigl(\mathbf c^{(1)},\dots,\mathbf c^{(r)}\bigr)
	=\boldsymbol{\rho}_{\mathrm{all}}+\mathsf T_{\mathrm{big}}\boldsymbol{\eta}_{\mathrm{all}},
	\]
	where
	\[
	\mathsf T_{\mathrm{big}}
	=
	\begin{bmatrix}
	\mathsf T^{(1)} & 0 & \cdots & 0 \\
	\star_{21} & \mathsf T^{(2)} & \cdots & 0 \\
	\vdots & \ddots & \ddots & \vdots \\
	\star_{r1} & \cdots & \star_{r,r-1} & \mathsf T^{(r)}
	\end{bmatrix}
	\]
	is \emph{block lower triangular} with \emph{diagonal blocks} $\mathsf T^{(j)}$.
	The starred blocks capture causal overlaps but do not destroy triangularity.
	Since each $\mathsf T^{(j)}$ is nonsingular, $\mathsf T_{\mathrm{big}}$ is also nonsingular.
	Moreover $\mathrm{Cov}(\boldsymbol{\eta}_{\mathrm{all}})\succ0$ because each component of $\boldsymbol{\eta}^{(j)}$ is drawn from nondegenerate Gaussians $(\phi_1,\nu,v)$.
	Therefore
	\[
	\mathrm{Cov}(\mathbf C)
	=\mathrm{Cov}(\boldsymbol{\rho}_{\mathrm{all}})+\mathsf T_{\mathrm{big}}\mathrm{Cov}(\boldsymbol{\eta}_{\mathrm{all}})\mathsf T_{\mathrm{big}}^\top
	\ \succ\ 0.
	\]
	\medskip
	\noindent\textbf{Part IV: Algebraic-variety step and almost-sure full row rank.}
	Let
	\[
	\boldsymbol{\zeta}_{\mathrm{all}}\;\triangleq\;\mathrm{col}\Bigl(\,\phi_{t_k},\ u_{t_k:t_k+T_{\mathrm{ini}}+T-1},\ y_{t_k:t_k+T_{\mathrm{ini}}-1}\,\Bigr)_{k=1}^r
	\]
	denote the stacked random vector on which $\mathbf C$ depends (so we may write $\mathbf C=\mathbf C(\boldsymbol{\zeta}_{\mathrm{all}})$).
	Set
	\[
	\Delta(\boldsymbol{\zeta}_{\mathrm{all}})\;=\;\det(\mathbf C).
	\]
	Then $\Delta$ is a polynomial in the entries of $\boldsymbol{\zeta}_{\mathrm{all}}$, and
	\[
	\{\mathrm{rank}\,\Xi^{(N)}<r\}\ \subseteq\ \{\Delta(\boldsymbol{\zeta}_{\mathrm{all}})=0\}.
	\]

	By our column construction, each entry of $\mathbf C$ depends on a \emph{distinct} scalar coordinate of $\boldsymbol{\zeta}_{\mathrm{all}}$. Hence these entries can be prescribed independently. In particular, choose a deterministic point $\boldsymbol{\zeta}_{\mathrm{all}}^*$ such that $\mathbf C(\boldsymbol{\zeta}_{\mathrm{all}}^*)=I_r$; then $\Delta(\mathbf C(\boldsymbol{\zeta}_{\mathrm{all}}^*))=1\neq0$, so $\Delta$ is not the zero polynomial.

	Therefore, the set
	\[
	\mathcal M\;=\;\{\boldsymbol{\zeta}_{\mathrm{all}}:\ \Delta(\boldsymbol{\zeta}_{\mathrm{all}})=0\}\ \subset\mathbb R^{r^2}
	\]
	is an algebraic variety of Lebesgue measure zero. Let $\lambda$ denote Lebesgue measure on $\mathbb R^{r^2}$ and let $\mu$ be the measure induced by~$\boldsymbol{\zeta}_{\mathrm{all}}$. By construction, $\boldsymbol{\zeta}_{\mathrm{all}}$ is jointly Gaussian with a nonsingular covariance matrix; hence $\mu$ admits an everywhere positive density with respect to~$\lambda$, i.e., $\mu\ll\lambda$. Consequently every $\lambda$-null set is also $\mu$-null, and since $\lambda(\mathcal M)=0$ we have $\mu(\mathcal M)=0$.

	Thus the rank-deficiency event has probability zero, i.e.,
	\[
	\mathbb P\!\bigl(\mathrm{rank}\,\Xi^{(N)}=r\bigr)=1.
	\]

    Moreover, when the trajectory library consists of multiple independent trajectories, Assumption~\ref{assump:multiple_trajectories} ensures that the random vectors forming each column of $\Xi^{(N)}$ are jointly Gaussian with a nonsingular covariance matrix.
    Applying the same argument as above, we again conclude that $\Xi^{(N)}$ has full row rank almost surely.

    \noindent\textbf{Non-Gaussian extension.}
	The argument above did not rely on any particular Gaussian property beyond 
	absolute continuity with respect to the Lebesgue measure. 
	Hence the almost-sure full row rank conclusion extends to a broader class of noise distributions.

	For the case where the trajectory library is constructed from a single trajectory, recall that $\mathbf C=\boldsymbol{\rho}_{\mathrm{all}}+\mathsf T_{\mathrm{big}}\boldsymbol{\eta}_{\mathrm{all}}$, 
	where $\mathsf T_{\mathrm{big}}$ is block lower triangular with full row rank and thus invertible. 
	By Assumption~\ref{assump:nonGaussian}, the support of $\boldsymbol{\eta}_{\mathrm{all}}$ has nonempty interior. 
	Then, by invertibility of $\mathsf T_{\mathrm{big}}$, the support of $\boldsymbol{\zeta}_{\mathrm{all}}$ also has nonempty interior. 
	Consequently, the measure induced by $\boldsymbol{\zeta}_{\mathrm{all}}$ is absolutely continuous with respect to Lebesgue measure. 
	Since the rank-deficiency event is contained in the algebraic variety $\mathcal M$
	which has Lebesgue measure zero, it follows that $\mathbb P(\mathcal M)=0$. 
	Therefore,
	\[
	\mathbb P\bigl(\mathrm{rank}\,\Xi^{(N)}=r\bigr)=1.
	\]
    The same argument applies when the trajectory library consists of multiple independent trajectories, completing the proof.
	
\end{proof}
\end{rese}

\bibliography{ref}
\bibliographystyle{ieeetr}
\vspace{-1.3cm}
\begin{biographynophoto}{Jiayun Li} received her Bachelor of Engineering degree from Department of Automation, Tsinghua University in 2022. She is currently a Ph.D. candidate in the Department of Automation, Tsinghua University. Her research interests include system identification, learning-based control.
\end{biographynophoto}
\vspace{-1cm}
\begin{biographynophoto}{Yilin Mo} is an Associate Professor in the Department of Automation, Tsinghua University. He received his Ph.D. In Electrical and Computer Engineering from Carnegie Mellon University in 2012 and his Bachelor of Engineering degree from Department of Automation, Tsinghua University in 2007. Prior to his current position, he was a postdoctoral scholar at Carnegie Mellon University in 2013 and California Institute of Technology from 2013 to 2015. He held an assistant professor position in the School of Electrical and Electronic Engineering at Nanyang Technological University from 2015 to 2018. His research interests include secure control systems and networked control systems, with applications in sensor networks and power grids.
\end{biographynophoto}

\end{document}